\documentclass[reqno,11pt]{amsart}
\usepackage[utf8]{inputenc}
\usepackage{etoolbox}
\usepackage{amsmath,amsfonts,amssymb,amsthm,graphicx,color,bbm}
\usepackage{mathrsfs}
\usepackage{yhmath}
\usepackage{xcolor}
\usepackage[colorlinks=true, pdfstartview=FitV, linkcolor=blue, 
            citecolor=blue, urlcolor=blue]{hyperref}

\usepackage{tikz}
\usepackage{float, ifthen}
\usepackage{pgfplots}
\pgfplotsset{compat=newest}
\pgfplotsset{colormap={violet}{rgb255=(25,25,122) rgb255=(238,140,238) color=(white)}}
\usetikzlibrary{math, calc}

 \allowdisplaybreaks
\numberwithin{equation}{section}

\def\be{\begin{equation}}
\def\ee{\end{equation}}
\def\bea{\begin{eqnarray}}
\def\eea{\end{eqnarray}}
\font\script=rsfs10 at 12pt
\def\eps{\varepsilon}

\def\A{{\mathcal A}}

\def\C{{\mathcal C}}
 
\def\G{{\mathcal G}} 

\def\L{{\mbox{\script L}\,}}

\def\R{{\mathbb R}}

\def\P{{\mathbb P}} 
\def\E{{\mathbb E}} 

\def\N{{\mathbb N}}
\def\Z{{\mathbb Z}}
\def\bS{{\mathbb S}} 
\def\H{{\mathbb H}} 
\def\cH{{\mathcal H}} 
\def\One{{\mathbbm 1}} 

\usepackage{mathtools}
\mathtoolsset{showonlyrefs}

\newcommand{\M}{{\mathcal M}}

\renewcommand{\(}{\left(}
\renewcommand{\)}{\right)}

\newcommand{\vol}{\mathrm{vol}} 
\newcommand{\supp}{\mathrm{supp}\ } 

\def\bal{\begin{aligned}}
\def\eal{\end{aligned}}

\def\proofof#1{\begin{proof}[Proof of #1]}

\def\part#1#2{\par\noindent{\underline{\it Part~#1.}}\emph{ #2}\\}

\def\({\left(}
\def\){\right)}

\DeclareMathOperator{\diam}{\rm diam}

\newcommand*\di{\mathop{}\!\mathrm{d}}

\def\XXint#1#2#3{{\setbox0=\hbox{$#1{#2#3}{\int}$} \vcenter{\vspace{-1pt}\hbox{$#2#3$}}\kern-.5\wd0}}

\theoremstyle{plain}
\newtheorem{lemma}{Lemma}[section]
\newtheorem{prop}[lemma]{Proposition}
\newtheorem{theorem}[lemma]{Theorem}
\newtheorem{example}[lemma]{Example}
\newtheorem{corol}[lemma]{Corollary}
\newtheorem{defin}[lemma]{Definition}
\newtheorem{remark}[lemma]{Remark}
\newtheorem*{remark*}{Remark}
\newtheorem*{notation*}{Notation}
\newtheorem*{hypothesis*}{Hypothesis}

\newcounter{mt}

\begin{document}

\title[Asymptotic quantization on Riemannian manifolds via covering growth estimates]{Asymptotic quantization of measures on Riemannian manifolds via covering growth estimates}
\author{Ata Deniz Ayd{\i}n}
\author{Mikaela Iacobelli}
\address{ETH Z\"urich, Department of Mathematics, R\"amistrasse 101, 8092 Z\"urich, Switzerland.}
\email[Ata Deniz Ayd{\i}n]{deniz.aydin@math.ethz.ch}
\email[Mikaela Iacobelli]{mikaela.iacobelli@math.ethz.ch}

\begin{abstract}
The quantization problem looks for best approximations of a probability measure on a given metric space by finitely many points, where the approximation error is measured with respect to the Wasserstein distance. On particular smooth domains, such as $\R^d$ or complete Riemannian manifolds, the quantization error is known to decay polynomially as the number of points is taken to infinity, provided the measure satisfies an integral condition which controls the amount of mass outside compact sets. On Riemannian manifolds, the existing integral condition involves a quantity measuring the growth of the exponential map, for which the only available estimates are in terms of lower bounds on sectional curvature.

In this paper, we provide a more general integral condition for the asymptotics of the quantization error on Riemannian manifolds, given in terms of the growth of the covering numbers of spheres, which is purely metric in nature and concerns only the large-scale growth of the manifold. We further estimate the covering growth of manifolds in two particular cases, namely lower bounds on the Ricci curvature and geometric group actions by a discrete group of isometries. 
These estimates can themselves generalize beyond manifolds, and hint at a future treatment of asymptotic quantization also on non-smooth metric measure spaces.
\end{abstract}

\maketitle

\section{Introduction}

\subsection{General overview}

The quantization problem for measures on metric spaces involves the approximation of finite measures by finitely many points. 
Given a finite Borel measure $\mu$ on a complete and separable metric space $X$, the quantization problem of order $p \in [1,\infty]$ involves the minimization of the \emph{$p$-Wasserstein distance} between $\mu$ and a measure supported on at most $N$ points:
\[ e_{N,p}(\mu) = \inf_{\# \supp\mu_N \leq N} W_p(\mu, \mu_N). \]
This problem admits various equivalent formulations; in particular, it reduces to a problem over the support of the approximating measure $\mu_N$. More specifically, it is equivalent to minimizing the following functional
\[ e_p(\mu; S) := \|d(\cdot, S)\|_{L^p(X; \mu)} = \left\{ \begin{matrix} \sqrt[p]{ \int_X d(x, S)^p \di\mu(x) }, & p < \infty; \\ \sup_{x \in \supp\mu} d(x,S), & p = \infty \end{matrix} \right. \]
over all sets $S \subseteq X$ of cardinality at most $N$:
\[ e_{N,p}(\mu) = \inf_{\# S \leq N} e_p(\mu; S). \]
For example, suppose $X$ represents a geographical region, and $\mu$ describes the distribution of a population on the region. Then the points in $S$ are chosen to be optimally reachable by the population in some sense. For $p = 1$, $S$ is chosen such that the average distance to $S$ is minimized, and for $p = \infty$, $S$ is chosen such that the maximal distance to $S$ is minimized; other values of $p$ interpolate between these problems.

In this paper, we will work only with this set distance formulation and the case $p < \infty$, and denote the $p$th powers of $e_p(\mu; S)$ and $e_{N,p}(\mu)$ by $V_p(\mu; S)$ and $V_{N,p}(\mu)$ respectively. For reference on other equivalent formulations, see e.g. Graf and Luschgy \cite[Sec. 3]{quantbook}. 

We will be focusing on the asymptotics of the quantization error as $N \to \infty$. Before reviewing existing results on asymptotic quantization on $\R^d$ and Riemannian manifolds, we first give a general survey of the history and applications of the quantization problem.

\subsection{Historical notes and applications}
The origins of the quantization problem lie in the field of signal processing, with the goal of efficiently compressing analog signals (such as sounds or images) into digital ones taking values in a finite set. The problem of \emph{scalar quantization}, quantization on $\R$, has been studied since the 1940s in information theory e.g. by Bennett \cite{bennett} and by Oliver, Pierce and Shannon \cite{PCM}. The quantization problem on $\R^d$, under the name of \emph{vector quantization}, has been studied in the signal processing community starting from the 1970s, e.g. by Gersho \cite{gersho}, and afterwards by Zador \cite{zador}, Bucklew and Wise \cite{buckwise} and Graf and Luschgy \cite{quantbook} who completed the proof of Zador's theorem on $\R^d$.
For a detailed survey of quantization from the point of view of signal processing and information theory, we refer to Gray and Neuhoff \cite{grayneuhoff}. See also Pag\`es \cite{pagesintro} for an introduction to quantization and its applications in numerics.

Independently, Steinhaus \cite{steinhaus} already in 1956 considered the quantization problem on $\R^d$ in a different formalism, and afterwards in 1959, L. Fejes T\'oth \cite{Fejes-Toth:1959vj} demonstrated the asymptotic optimality of hexagonal lattices for quantization on $\R^2$. This asymptotic optimality and stability of the hexagonal lattice is also observed in similar geometric problems such as the sphere packing problem.
This result of Fejes T\'oth has been the subject of extensive study and generalization, see e.g. \cite{Gruber:1999tm, GToth:2001, Boroczky:2010}, the extension to compact Riemannian 2-manifolds \cite{Gruber2001} as well as more recent results by \cite{Caglioti:2018, Iacobelli:2018tz, Bourne:2021ue}.
There also exist more purely geometric applications of quantization, such as the approximation of convex bodies by polytopes \cite{Bourgain:1989aa, Gruber2001, Gruber2004, discapprox}, and \emph{Alexandrov's problem} of constructing a convex body or surface with prescribed Gaussian curvature \cite{Bonk:2003vi, Merigot:2016us}. These applications provide further motivation for the study of quantization for Riemannian manifolds.

The quantization problem has also been studied under different names, e.g. the \emph{optimal location problem} \cite{Bouchitte:2011tp, Bourne:2015, Bourne:2021ue}, \emph{centroidal Voronoi tessellations} \cite{cvt, cvt:2005}, or the (more restrictive) problem of \emph{$k$-means clustering} in statistics \cite{MacQueen:1967}. 
A related problem is \emph{empirical quantization} or \emph{uniform quantization}, in which one only considers the $p$-Wasserstein distance between a measure $\mu$ and empirical measures (measures of the form $\frac{1}{N} \sum_{i=1}^N \delta_{x_i}$). This problem also admits applications in e.g. image analysis, and has been studied in e.g. \cite{chevallier_2018, Kloeckner:2020, Merigot:2021} in the \emph{deterministic} setting, where one seeks to find the optimal $\{x_i\}_{i=1}^N$, and in e.g. \cite{hochbaum_steele_1982, Graf:2002ww, Dereich:2013, Garcia-Trillos:2015wv, ambrosio2019finer, Ambrosio:2019us, Ambrosio:2019wy, Benedetto:2020aa, Benedetto:2021aa} in the \emph{stochastic} setting (also referred to as \emph{random matching}), where the points $\{x_i\}_{i=1}^N$ are generated independently from a probability distribution, usually $\mu$ itself.

Recent directions in the study of quantization include the calculus of variations/$\Gamma$-convergence approach applied in \cite{Bouchitte:2011tp, Bourne:2021ue} for more general entropy as well as cardinality constraints, and the gradient flow approach taken by \cite{Caglioti_2015, Caglioti:2018, Iacobelli:2018tz} for quantization on $\R$ and $\R^2$, leading to the consideration of \emph{ultrafast diffusion equations} \cite{Iacobelli:2019aa, Iacobelli:2019wz}.

On Riemannian manifolds, the asymptotics of the quantization problem have been investigated by Gruber \cite{Gruber2001, Gruber2004} and Kloeckner \cite{discapprox} for compactly supported measures, and generalized to the non-compact case by Iacobelli \cite{iacasym}. See also the recent work by Le Brigant and Puechmorel \cite{Brigant2019uo, Brigant2019vo} introducing an algorithm for finding optimal quantizers for fixed $N$ on Riemannian manifolds.
We will now focus on existing results for the asymptotics of quantization on Euclidean space and Riemannian manifolds in more detail, and afterwards state our new contributions.

\subsection{Asymptotics of quantization on $\R^d$ and Riemannian manifolds}
In particular domains, such as $\R^d$, Riemannian manifolds and also fractal domains, the quantization error $V_{N,p}(\mu)$ has been shown to decay on the order of $N^{-p/d}$ as $N \to \infty$, where $d$ represents the dimensionality of the domain. 
Heuristically, this means that in order to decrease the error $e_{N,p}(\mu)$ by a factor of $k$, one would need to scale $N$ by $k^d$. This exponential dependence on the dimension is also present in related problems, such as \emph{optimal matching} (see e.g. survey given in \cite[\S 1.2.4]{discapprox}). More generally, it is related to the so-called \emph{curse of dimensionality} phenomenon in approximation, where the number or complexity of parameters required to achieve a given error threshold grows exponentially with the dimension of the data (cf. \cite[Thm. 1.1]{Luschgy:2015} in relation to quantization, and \cite[\S 5.2.1]{clustering} for a broader background).

For measures on $\R^d$, \emph{Zador's theorem} provides the following precise expression for the asymptotics of the quantization error:
\begin{equation}\label{eq:zadorasym}
\lim_{N \to \infty} N^{p/d} V_{N,p}(\mu) = Q_p([0,1]^d) \left( \int_{\R^d} \rho(x)^{\frac{d}{d+p}} \di x \right)^{\frac{d+p}{d}},
\end{equation}
where $\rho$ is the density of the absolutely continuous component of $\mu$ with respect to the Lebesgue measure, and $Q_p([0,1]^d) \in (0,\infty)$ is a constant referred to as the \emph{quantization coefficient} of the $d$-dimensional unit cube. The most general form of this theorem is due to Graf and Luschgy \cite[Thm. 6.2]{quantbook}, and is first proven for measures of compact support, and then generalized to measures which admit finite moments of higher order:
\begin{equation}\label{eq:momentcond}
\int_{\R^d} \|x\|^{p+\delta} \mu(x) < \infty \quad \text{for some } \delta > 0. 
\end{equation}
This condition implies also the finiteness of the right-hand side in \eqref{eq:zadorasym} \cite[Rem. 6.3]{quantbook}.
The extension from the compact to the noncompact case is facilitated by \emph{Pierce's lemma}, which provides a universal upper bound on the quantization error:

\begin{lemma}[Pierce {\cite[Thm. 1]{pierce}}, Graf and Luschgy {\cite[Lem. 6.6]{quantbook}}]
Let $d \in \N$, $p \in [1,\infty)$, $\delta > 0$.
Then there exist constants $C > 0$ and $N_0 \in \N$ depending on $d$, $p$ and $\delta$ such that, for any probability measure $\mu$ on $\R^d$, 
\[ N^{p/d} V_{N,p}(\mu) \leq C \int_{\R^d} (1+\|x\|^{p+\delta}) \di \mu(x) , \quad \textup{for all } N \geq N_0. \]
\end{lemma}

Zador's theorem extends to measures on Riemannian manifolds by a similar argument to $\R^d$. The analogous statement to \eqref{eq:zadorasym}, with the Lebesgue measure replaced with the Riemannian volume form, has first been proven by Kloeckner \cite{discapprox} for compactly supported measures, for which no integral condition is necessary. This result was afterwards extended by Iacobelli \cite{iacasym} to measures satisfying a more refined integrability condition depending also on the curvature of the manifold. The analogous statement to Pierce's lemma on manifolds is given as follows:

\begin{theorem}[Iacobelli {\cite[Thm. 3.1]{iacasym}}]\label{thm:iacub}
Let $M$ be a complete connected $d$-dimensional Riemannian manifold, $x_0 \in M$, $p \in [1,\infty)$, $\delta > 0$.
Then there exist constants $C > 0$, $N_0 \in \N$ depending on $M$, $x_0$, $p$ and $\delta$ such that, for any probability measure $\mu$ on $M$, 
\[ N^{p/d} V_{N,p}(\mu) \leq C \int_{M} \left(1+d(x,x_0)^{p+\delta}+\A_{x_0}(d(x,x_0))^p \right) \di \mu(x) , \quad \textup{for all } N \geq N_0, \]
where $\A_{x_0}$ measures the size of the differential of the exponential map at $x_0$, restricted to the sphere $\bS^{d-1}_r \subset T_{x_0} M$ of radius $r$ centered at $0$:
\begin{equation}\label{eq:ax0}
\A_{x_0}(r) := \sup_{\substack{v \in \bS^{d-1}_r, \\ w \in T_v \bS^{d-1}_r, \\ \|w\|_v = r}} \left\| d_v \exp_{x_0}[w] \right\|_{\exp_{x_0}(v)}.
\end{equation}
\end{theorem}

The generalization of Zador's theorem for non-compactly supported measures is then deduced from the finiteness of the integral expression:

\begin{theorem}[Iacobelli {\cite[Thm. 1.4]{iacasym}}]
Let $M$ be a complete connected $d$-dimensional Riemannian manifold, $p \in [1,\infty)$. Let $\mu$ be a probability measure on $M$ such that
\[ \int_{M} \left(d(x,x_0)^{p+\delta}+\A_{x_0}(d(x,x_0))^p \right) \di \mu(x) < \infty \]
for some $x_0 \in M$ and $\delta > 0$. Then
\begin{equation}
\lim_{N \to \infty} N^{p/d} V_{N,p}(\mu) = Q_p([0,1]^d) \left( \int_{M} \rho(x)^{\frac{d}{d+p}} \di \vol_M(x) \right)^{\frac{d+p}{d}} < \infty,
\end{equation}
where $\rho$ is the density of the absolutely continuous component of $\mu$ with respect to the Riemannian volume form $\vol_M$ on $M$.
\end{theorem}

The quantity $\A_{x_0}(r)$ can be estimated in terms of sectional curvature bounds, yielding the following corollary:

\begin{corol}[Iacobelli {\cite[Cor. 1.6]{iacasym}}]
Let $M$ be a complete connected $d$-dimensional Riemannian manifold, with sectional curvature bounded from below by $-\kappa^2$ for some $\kappa \geq 0$, $p \in [1,\infty)$. Let $\mu$ be a probability measure on $M$ such that
\[ \int_{M} (d(x,x_0)^{p+\delta}+ e^{\kappa p d(x,x_0))} ) \di \mu(x) < \infty \]
for some $x_0 \in M$ and $\delta > 0$. Then
\begin{equation}
\lim_{N \to \infty} N^{p/d} V_{N,p}(\mu) = Q_p([0,1]^d) \left( \int_{M} \rho(x)^{\frac{d}{d+p}} \di \vol_M(x) \right)^{\frac{d+p}{d}}.
\end{equation}
\end{corol}

In particular, Pierce's lemma and Zador's theorem hold for manifolds of nonnegative sectional curvature in the same form as for Euclidean space.
However, for manifolds of constant negative curvature this exponential moment condition is sharp: \cite[Thm. 1.7]{iacasym} constructs a (singular) measure $\mu$ on the hyperbolic plane $\H^2$ for which all polynomial moments $\int_{\H^2} d(x,x_0)^p \di \mu(x)$ are finite but $\lim_{N \to \infty} N^{p/d} V_{N,p}(\mu) = \infty$.

\subsection{Definitions and main results}

In this paper, we demonstrate that an exponential moment condition is not strictly necessary even for nonpositively curved manifolds when the sectional curvature is not constant. We achieve this by replacing the local quantity $\A_{x_0}$ with global covering numbers capturing the large-scale growth of the manifold, which can be defined in a general metric space setting:

\begin{defin}
Let $X$ be a metric space.
Given a compact subset $A \subseteq X$ and $N \in \N$, the $N$th \emph{covering radius} of $A$ is the minimal $r > 0$ such that $A$ can be covered by $N$ open balls of radius $r$:
\begin{align*} 
r_N(A) := & \inf\left\{ r > 0 \mid \exists \{x_i\}_{i=1}^N \subseteq X, A \subseteq \bigcup_{i=1}^N B_r(x_i) \right\}.
\end{align*}
Fix $d > 0$, and let $f \colon [0,\infty) \to [0,\infty)$ be monotone nondecreasing. $X$ is said to have \emph{$O(f)$ covering growth} (of dimension $d$) around $x_0 \in X$ if
\begin{equation}\label{eq:covgrow} k r_{k^{d-1}}(\partial B_R(x_0)) \leq f(R) , \quad \textup{for all } k \in \N, R > 0. \end{equation}
\end{defin}

\begin{remark*}
For Riemannian manifolds, $d$ will be the same as the dimension of the manifold. In general, the hypothesis that $\sup_{k \in \N} k r_{k^{d-1}}(\partial B_R(x_0)) < \infty$ implies that spheres $\partial B_R(x_0)$ are at most $d-1$-dimensional in the sense of box-counting dimension (see e.g. \cite[Sect. 11.2]{quantbook}). 

For our purposes, it suffices for \eqref{eq:covgrow} to be satisfied for $k \geq k_0$ and $R \geq R_0$ for constants $k_0, R_0 > 0$, but not when $k \geq k_0(R)$ dependent on $R$. Up to a factor of $2$, one can equivalently require that $N^{\frac{1}{d-1}} r_N(\partial B_R(x_0)) \leq f(R)$ for all $N \in \N$, since for each $N \in \N$ there exists $k \in \N$ such that $k^{d-1} \leq N \leq (k+1)^{d-1} \leq 2^{d-1} k^{d-1}$. 
\end{remark*}


As detailed in Section \ref{sect:riclb}, Riemannian manifolds of nonnegative Ricci curvature including Euclidean spaces have $O(R)$ covering growth, while e.g. hyperbolic spaces only have $O(e^{\alpha R})$ covering growth for some $\alpha > 0$. As a special case, Example \ref{ex:twodim} shows that the covering growth of two-dimensional Riemannian manifolds coincides exactly with the growth of the perimeters of spheres, although this correspondence does not extend to higher dimensions.

In general, as demonstrated in Example \ref{ex:ax0covnum}, any complete connected Riemannian manifold has $O(\A_{x_0})$ covering growth around $x_0$. 
As such, Theorem \ref{thm:iacub} admits the following generalization, valid in the setting of geodesic spaces:

\begin{theorem}\label{thm:quantub}
Let $X$ be a complete geodesic space. Fix a point $x_0 \in X$, and suppose $X$ has $O(f)$ covering growth of dimension $d$ around $x_0$.

Let $p \in [1,\infty)$, and let $\mu$ a finite Borel measure on $X$ with finite moments up to order $p+\delta$ for some $\delta > 0$. Then there exists a constant $C > 0$ independent of $\mu$, for which
\[ k^p V_{k^d+1,p}(\mu) \leq C \int_X \left(1+d(x, x_0)^{p+\delta}+f(d(x,x_0))^p \right) \di \mu(x), \quad \text{for all } k \in \N. \]
Consequently, 
\[ \limsup_{N\to\infty} N^{p/d} V_{N,p}(\mu) \leq C \int_X \left(1+d(x, x_0)^{p+\delta}+f(d(x,x_0))^p \right) \di \mu(x). \]
\end{theorem}

As a direct consequence, on Riemannian manifolds we deduce the following extension of Zador's theorem:

\begin{theorem}\label{thm:quantasym}
Let $M$ be a complete connected $d$-dimensional Riemannian manifold. Fix a point $x_0 \in M$, and suppose $M$ has uniform $O(f)$ covering growth (of dimension $d$) around $x_0$.

Let $p \in [1,\infty)$, and let $\mu$ be a finite Borel measure on $M$ such that, for some $\delta > 0$,
\[ \int_M \left(d(x, x_0)^{p+\delta}+f(d(x,x_0))^p \right) \di \mu(x) < \infty. \]
Then 
\begin{equation}\label{eq:riemasym}
\lim_{N \to \infty} N^{p/d} V_{N,p}(\mu) = Q_p([0,1]^d) \left( \int_{M} \rho(x)^{\frac{d}{d+p}} \di x \right)^{\frac{d+p}{d}},
\end{equation}
where $\rho$ is the density of the absolutely continuous component of $\mu$ with respect to the Riemannian volume form on $M$, and $Q_p([0,1]^d) \in (0,\infty)$ is again the quantization coefficient of the $d$-dimensional unit cube.
\end{theorem}

This is a strictly more general condition than the one expressed in terms of $\A_{x_0}$. To demonstrate this, we consider two special cases: manifolds with Ricci curvature bounded from below, and manifolds subject to a \emph{geometric action} of a discrete group of isometries. In particular, we show in Section \ref{sect:riclb} that if the manifold has nonnegative \emph{Ricci} curvature, then it has $O(R)$ covering growth and thus Pierce's lemma holds as in Euclidean space. In the latter case, if the group of isometries is finitely generated with polynomial growth, the covering growth of the manifold is also polynomial and hence a polynomial moment condition is sufficient for Zador's theorem to hold: 

\begin{prop}
Suppose $M$ is subject to the geometric action of a group $G$ of isometries, which is of \emph{polynomial growth} of order $\alpha$. 

Then fixing $x_0 \in M$, there exists $C > 0$ such that for $R > 0$ sufficiently large and $N \in \N$ arbitrary, the following bound holds for the covering growth of spheres in $M$:
\[ N r_N(\partial B_R(x_0))^{d-1} \leq C R^{\alpha+d-1}. \]
Thus $M$ has $O(R^{\frac{\alpha}{d-1}+1})$ covering growth around any point $x_0 \in M$.
\end{prop}

These results rely only on metric arguments and the \emph{Bishop-Gromov volume comparison theorem} for complete Riemannian manifolds, and could thus be extended to non-smooth metric measure spaces which satisfy the same volume comparison property. Specifically, the \emph{curvature-dimension spaces} introduced independently by Sturm \cite{sturm, sturmii} and by Lott and Villani \cite{LottVillani} are shown to satisfy the same volume comparison property as Riemannian manifolds (cf. \cite[Thm. 2.3]{sturmii}). 
While Theorem \ref{thm:quantasym} will not extend directly to curvature-dimension spaces, owing to the dependence on the locally Euclidean structure of manifolds, the general upper bound in Theorem \ref{thm:quantub} as well as the estimates given in Section \ref{sect:covest} will carry over thanks to the Bishop-Gromov theorem. Though we will only be focusing on Riemannian manifolds in this paper, we nevertheless present Theorem \ref{thm:quantub} in the general setting of geodesic spaces for the sake of future investigations of non-smooth settings.

\subsection{Organization}
Before proceeding with the proof of Theorems \ref{thm:quantub} and \ref{thm:quantasym}, in Section \ref{sect:prel} we cover preliminary definitions and propositions about the quantization error and covering numbers. We then prove Theorem \ref{thm:quantub} in Section \ref{sect:proofub} and Theorem \ref{thm:quantasym} in Section \ref{sect:proofasym}. We obtain estimates on the covering growth of manifolds in Section \ref{sect:covest}, treating the case of lower bounded Ricci curvature in Section \ref{sect:riclb} and geometric group actions in Section \ref{sect:cocompact}. Appendix \ref{app:mincon} discusses the relation between covering growth and the perimeters of spheres, and Appendix \ref{app:auxlem} contains auxiliary lemmas.

\section{Preliminaries}\label{sect:prel}

We first review a list of preliminary definitions and facts about the quantization error, as well as the covering radii of sets. The material in this section is valid in the setting of general \emph{Polish} metric spaces, i.e., separable metric spaces that admit an equivalent complete metric.

\begin{defin}
Let $X$ be a Polish metric space, $p \in [1,\infty)$. 

We denote by $\M^p_+(X)$ the set of finite positive Borel measures on $X$ with \emph{finite $p$th moments}: $\mu \in \M^p_+(X)$ if $\int_X d(x,x_0)^p \di \mu(x) < \infty$ for some ($\iff$ all) $x_0 \in X$.

Given $\mu \in \M^p_+(X)$, the \emph{quantization error} of a set $S \subseteq X$ of order $p$ with respect to $\mu$ is
\[ V_p(\mu; S) := \int_X d(x, S)^p \di\mu(x). \]
The $N$th \emph{quantization error} of $\mu$ of order $p$ is the infimum over all subsets of cardinality at most $N$:
\[ V_{N,p}(\mu) := \inf_{\# S \leq N} V_p(\mu; S). \]
The $p$th root of the quantization error is also denoted by $e_{N,p}(\mu)$. If $A \subseteq X$, we also denote $V^{(\mu)}_p(A; S) := V_p(\mu|_A; S)$ and likewise for $V^{(\mu)}_{N,p}$ and $e^{(\mu)}_{N,p}$.
\end{defin}

The quantization error of a finite linear combination of measures can be decomposed linearly:

\begin{lemma}\label{lem:quantadd}
Let $\mu_1, \ldots, \mu_k$ be finite Borel measures on a Polish metric space $X$, with finite $p$th moments for $p < \infty$. 

Set $\mu := \sum_{i=1}^k \lambda_i \mu_i$, $\lambda_1, \ldots, \lambda_k > 0$. 
Then for any $N, N_1, \ldots, N_k \in \N$ such that $N \geq \sum_{i=1}^k N_i$,
\[ V_{N,p}(\mu) \leq \sum_{i=1}^k \lambda_i V_{N_i,p}(\mu_i). \]
\end{lemma}
\begin{proof}
For each $i = 1, \ldots, k$, let $S_i$ be an arbitrary subset of $X$ with cardinality at most $N_i$. Then $S := \bigcup_{i=1}^k S_i$ has cardinality at most $N$, and
\begin{align*}
 V_{N,p}(\mu) & \leq \int_X d(x, S)^p \di \mu(x) = \sum_{i=1}^k \lambda_i \int_X d(x, S)^p \di \mu_i(x) \leq \sum_{i=1}^k \lambda_i \int_X d(x, S_i)^p \di \mu_i(x).
\end{align*}
Taking the infimum over all such $S_i$ yields the statement.
\end{proof}

For every $\mu \in \M^p_+(X)$, $V_{N,p}(\mu) \to 0$ as $N \to \infty$ (see e.g. \cite[Lem. 6.1]{quantbook}). When $\mu$ is compactly supported, this follows by taking an $\varepsilon$-cover of $\supp\mu$ for $\varepsilon > 0$ arbitrarily small. The general case follows from \emph{Ulam's lemma}, valid for Polish spaces, which we cite for later use:

\begin{lemma}[Ulam's lemma]
Let $X$ be a Polish space, $\nu$ a finite Borel measure on $X$. Then for any $\varepsilon > 0$, there exists $K \subset X$ compact such that $\nu(X \setminus K) < \varepsilon$.
\end{lemma}

In particular, for $\mu \in \M^p_+(X)$, for any $\varepsilon > 0$ there exists $K \subseteq X$ compact such that $\int_{X \setminus K} d(x, x_0)^p \di \mu(x) < \varepsilon$. For a proof of Ulam's lemma, refer to Dudley \cite[7.1.4]{dudley}.

The rate of convergence of the quantization error to zero can be further quantified by the following definition:

\begin{defin}
Let $X$ be a Polish metric space, $p \in [1,\infty)$, $d > 0$, $\mu \in \M^p_+(X)$.

The \emph{upper} resp. \emph{lower quantization coefficient} of $\mu$ of order $p$ and dimension $d$ is
\[ \overline{Q}_{p,d}(\mu) := \limsup_{N\to\infty} N^{p/d} V_{N,p}(\mu); \quad \underline{Q}_{p,d}(\mu) := \liminf_{N\to\infty} N^{p/d} V_{N,p}(\mu). \]
When the upper and lower limits coincide, the limit is denoted simply by $Q_{p,d}(\mu)$ and called the \emph{quantization coefficient} of $\mu$.
\end{defin}

We will only need the following subadditivity property of upper quantization coefficients, which is also used implicitly in proofs of the Zador theorem (cf. \cite[Lem. 6.5, 6.8]{quantbook}):

\begin{lemma}\label{lem:quantcoeffprops}
Let $\mu_1, \mu_2 \in \M^p_+(X)$, $p < \infty$. Set $\mu := \mu_1 + \mu_2$. Let $d > 0$, and fix $t_1, t_2 \in (0,1)$ such that $t_1 + t_2 = 1$. Then
\begin{align*} 
\overline{Q}_{p,d}(\mu) & \leq t_1^{-p/d} \overline{Q}_{p,d}(\mu_1) + t_2^{-p/d} \overline{Q}_{p,d}(\mu_2).
\end{align*}
\end{lemma}
\begin{proof}
For each $N \geq \max\{t_1^{-1}, t_2^{-1}\}$, set $N_1 = N_1(N) := \lfloor t_1 N \rfloor \geq 1$ and $N_2 = N_2(N) := \lfloor t_2 N \rfloor \geq 1$. Then by Lemma \ref{lem:quantadd},
\[ N^{p/d} V_{N,p}(\mu) \leq N^{p/d} V_{N_1,p}(\mu_1) + N^{p/d} V_{N_2,p}(\mu_2) = \left( \frac{N_1}{N} \right)^{-p/d} N_1^{p/d} V_{N_1,p}(\mu_1) + \left( \frac{N_2}{N} \right)^{-p/d} N_2^{p/d} V_{N_2,p}(\mu_2). \]
Taking the limit supremum of both sides, noting that the limit supremum is subadditive and that $\lim_{N\to\infty} \frac{N_i}{N} = t_i$, yields the inequality.
\end{proof}

In particular, picking the proportions $t_i$ optimally, we have the following inequality:

\begin{lemma}\label{lem:quantcoeffadd}
Let $\mu_1$, $\mu_2$ be finite Borel measures on a complete metric space $X$, with finite $p$th moments for $p < \infty$. Set $\mu := \mu_1 + \mu_2$. Let $s > 0$, and set $r := \frac{d}{p+d}$; i.e., $\frac{1}{r p} = \frac{1}{p} + \frac{1}{d}$. Then
\begin{align*} 
\overline{Q}_{p,d}(\mu)^r & \leq \overline{Q}_{p,d}(\mu_1)^r + \overline{Q}_{p,d}(\mu_2)^r.
\end{align*}
\end{lemma}
\begin{proof}
If either summand is zero, say $\overline{Q}_{p,d}(\mu_2) = 0$, then by Lemma \ref{lem:quantcoeffprops}, for arbitrary $t_1 \in (0,1)$ we have that
\[ \overline{Q}_{p,d}(\mu) \leq t_1^{-p/d} \overline{Q}_{p,d}(\mu_1) \] 
Letting $t_1 \to 1^-$ yields $\overline{Q}_{p,d}(\mu) \leq \overline{Q}_{p,d}(\mu_1)$ and the desired inequality holds.

We can then assume both summands are nonzero. Set
\[ t_i := Z^{-1} \overline{Q}_{p,d}(\mu_i)^r; \quad Z = \overline{Q}_{p,d}(\mu_1)^r+\overline{Q}_{p,d}(\mu_2)^r. \]
Then $t_1, t_2 \in (0,1)$ with $t_1 + t_2 = 1$, so by Lemma \ref{lem:quantcoeffprops},
\begin{align*} 
\overline{Q}_{p,d}(\mu) & \leq Z^{p/d} \overline{Q}_{p,d}(\mu_1)^{1-rp/d} + Z^{p/d} \overline{Q}_{p,d}(\mu_2)^{1-rp/d} 
\\ & \leq Z^{p/d} \left[ \overline{Q}_{p,d}(\mu_1)^r +  \overline{Q}_{p,d}(\mu_2)^r \right]
\\ & = Z^{1+p/d} = Z^{1/r}
\end{align*}
since $1-\frac{rp}{d} = 1 - \frac{p}{d+p} = \frac{d}{d+p} = r$. This shows that 
\[ \overline{Q}_{p,d}(\mu)^r \leq Z = \overline{Q}_{p,d}(\mu_1)^r+\overline{Q}_{p,d}(\mu_2)^r.  \qedhere \]
\end{proof}

This bound is an application of the equality case of the generalized H\"older inequality. Analogous statements can be formulated for the quantization problem of order $\infty$, but we omit its treatment here.

\subsection{Covering numbers and radii}

In preparation for the upper bound on the quantization error, we also review the covering numbers and radii of sets in general metric spaces. 

\begin{defin}
Let $A$ be a nonempty (pre)compact subset of a complete metric space $X$. 
\begin{itemize}
\item The \emph{covering number} of $A$ of radius $r$ is the smallest number of open $r$-balls that cover $A$:
\[ N(A; r) := \min\left\{ N \in \N \mid \exists \{x_i\}_{i=1}^N \subseteq X, A \subseteq \bigcup_{i=1}^N B_r(x_i) \right\}. \]
\item The \emph{packing number} of $A$ of radius $r$ is the greatest number of disjoint open $r$-balls with centers in $A$:
\[ P(A; r) := \max\left\{ N \in \N \mid \exists \{x_i\}_{i=1}^N \subseteq A, B_r(x_i) \cap B_r(x_j) = \varnothing \ \forall i \neq j \right\}. \]
\item The $N$th \emph{covering radius} of $A$ is the minimal radius of a cover of $A$ by $N$ open balls:
\begin{align*} 
r_N(A) := & \inf\left\{ r > 0 \mid \exists \{x_i\}_{i=1}^N \subseteq X, A \subseteq \bigcup_{i=1}^N B_r(x_i) \right\} 
\\ = & \inf\{ r > 0 \mid N(A; r) \leq N \}. 
\end{align*}
\end{itemize}
\end{defin}

The covering radius can also be thought of as the $N$th quantization error of $A$ of order $p = \infty$; see Graf and Luschgy \cite[Sect. 10]{quantbook}. 

\begin{remark}
For $A$ compact, we have the inequalities $r_{N(A; r)}(A) \leq r$ and $N(A; r_N(A)) > N$ for all $r > 0$ and $N \in \N$. The former inequality follows by definition. For the latter, note that there exists no $r_N(A)$-cover of $A$ with $N$ elements: given any $r$-cover $S := \{x_i\}_{i=1}^N$ of $A$ with $N$ elements, the compactness of $A$ implies that $\sup_{x \in A} d(x, S) = \max_{x \in A} d(x, S) < r$ since $A$ is contained in the open $r$-neighborhood of $S$.

If $T \colon X \to Y$ is a $L$-Lipschitz map between metric spaces with $A \subseteq X$, then $N(T(A); L r) \leq N(A; r)$ and $r_N(T(A)) \leq L r_N(A)$, since for any $r$-cover $\{x_i\}_{i=1}^N$ of $A$, $\{T(x_i)\}_{i=1}^N$ is an $L r$-cover of $T(A)$ with at most $N$ elements. Moreover, equality holds if $T \colon X \to Y$ is a similarity transformation; i.e., $d_Y(T(x), T(x^\prime)) = L d_X(x, x^\prime)$.
\end{remark}

The following basic inequalities show that packing and covering numbers grow at the same rate:

\begin{lemma}
Let $A$ be a nonempty precompact subset of a complete metric space $X$. Then for any $r > 0$,
\[ N(A; 2r) \leq P(A; r) \leq N(A; r). \] 
\end{lemma}
\begin{proof}
Note firstly that by the total boundedness of $A$, $N(A; r) < \infty$ for any $r > 0$. 

We first prove the second inequality. Let $n := N(A; r)$ and let $\{x_i\}_{i=1}^n \subseteq X$ be an $r$-cover of $A$. 
Let $\{y_j\}_{j=1}^m \subseteq A$ be an arbitrary finite subset of $A$ with $m > n$. Since $\{x_i\}_{i=1}^n$ is an $r$-cover, for each $j \in \{1, \ldots, m\}$ there exists $i \in \{1, \ldots, n\}$ such that $y_j \in B_r(x_i)$. 

Then by the assumption that $m > n$, the pigeonhole principle implies the existence of $j \neq j^\prime$ such that $y_j, y_{j^\prime} \in B_r(x_i)$. But in that case, $x_i \in B_r(y_j) \cap B_r(y_{j^\prime})$, so $\{y_j\}_{j=1}^m$ cannot be an $r$-packing on $A$. Thus $P(A; r) \leq n$, in particular $P(A; r) < \infty$. This shows the second inequality.

Now setting $m := P(A; r) \leq n < \infty$, by the well-ordering principle there exists an $r$-packing $\{y_j\}_{j=1}^m \subseteq A$. Let $x \in A \setminus \{y_j\}_{j=1}^m$. Then by maximality, the set $\{y_j\}_{j=1}^m \cup \{x\}$ cannot also be an $r$-packing, so there exists $j \in \{1, \ldots, m\}$ such that $B_r(y_j) \cap B_r(x) \neq \varnothing$. In that case, we have $d(x, y_j) < 2r$ by the triangle inequality. This shows that $\{y_j\}_{j=1}^m$ is a $2r$-cover of $A$, yielding the bound $P(A; r) = m \geq N(A; 2r)$. This shows the first inequality.
\end{proof}

We will rely on the former inequality in order to bound covering numbers from above.
For reference, see e.g. Mattila \cite[{\S 5.3}]{Mattila:1995ub}; note however the additional factor of $2$ arising from the different convention to define coverings and packings in terms of closed instead of open balls.

\begin{figure}[H]
\centering
\includegraphics[scale=0.9]{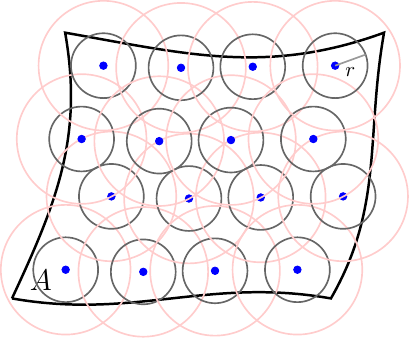}
\caption{A maximal $r$-packing on a region $A$ with cardinality $16$ yielding a $2r$-cover of $A$.}
\end{figure}

We also cite the following volumetric upper bound on covering/packing numbers:
 
\begin{lemma}\label{lem:supregbd}
Let $X$ be a metric space, $\nu$ a Borel measure on $X$. 

Let $A$ be a nonempty compact subset of $X$, and suppose there exist $\vartheta, d, r > 0$ such that
\[ \nu(B_r(x)) \geq \vartheta r^d, \quad \text{for all } x \in A. \]
Then
\[ N(A; 2r) \leq P(A; r) \leq \vartheta^{-1} \nu(A^{r}) r^{-d}, \]
where $A^r = \{ x \in X \mid d(x,A) < r \}$ is the open $r$-neighborhood of $A$.
\end{lemma}
\begin{proof}
Since $A$ is compact, $P(A; r) \leq N(A; r) < \infty$.
Take a maximal $r$-packing $\{x_i\}_{i=1}^N$ on $A$, with $N = P(A; r) \geq N(A; 2 r)$. Then the disjoint balls $(B_r(x_i))_{i=1}^N$ are contained in $A^r$, hence
\[ \nu(A^r) \geq \sum_{i=1}^N \nu(B_r(x_i)) \geq N \vartheta r^d. \]
Reordering yields the desired inequality.
\end{proof}


In Section \ref{sect:volest}, we will apply this basic volumetric bound in order to control the covering numbers of spheres in Riemannian manifolds.

\section{Proof of main results}

We now prove the Pierce-type upper bound on the quantization error of measures on complete Riemannian manifolds, and directly deduce the generalization of Zador's theorem for non-compactly supported measures.

\subsection{Proof of the upper bound}\label{sect:proofub}

\begin{proof}[Proof of Theorem \ref{thm:quantub}]
We obtain the upper bound by reduction to Pierce's lemma on $\R$. Take the distance function $d_{x_0} = d(\cdot, x_0) \colon X \to [0,\infty)$ and the pushforward measure $\mu_1 := (d_{x_0})_\# \mu$. We first quantize $\mu_1$ with $k$ radii $0 < R_1 < \ldots < R_k < \infty$, and then cover each sphere $\partial B_{R_i}(x_0)$ with $k^{d-1}$ points in order to quantize $\mu$ with $N := k^d + 1$ points.
 
\begin{figure}[H]
\centering
\includegraphics[scale=0.9]{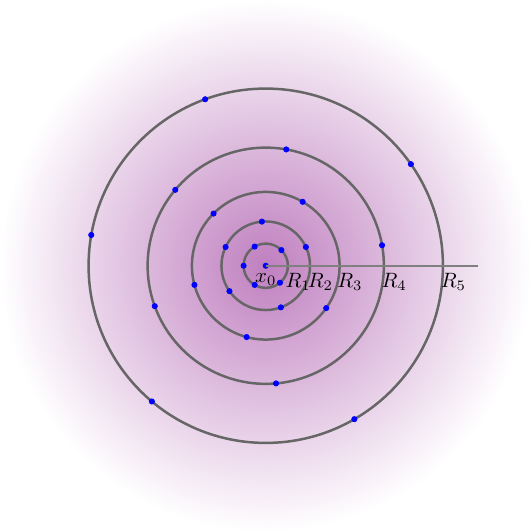}
\caption{Quantization of a non-uniform distribution on $\R^2$ with $N = 26$, obtained by covering 5 concentric circles with 5 points each.}
\end{figure}
 
\begin{notation*}
We will denote the ball $B_R(x_0)$ simply by $B(R)$, the sphere $\partial B_R(x_0)$ by $S(R)$, and the annulus $B_{R_2}(x_0) \setminus B_{R_1}(x_0)$ by $A[R_1, R_2)$ (in particular, $A[0,R) = B(R)$ and $A[R,\infty) = B(R)^c$).
\end{notation*}

Let $k \geq 1$. Choose $0 = R_0 < R_1 < \ldots < R_k < R_{k+1} = \infty$ arbitrary, so that $X$ can be partitioned as
\[ X = \bigsqcup_{i=0}^k X_i; \quad X_i := A[R_i, R_{i+1}) = \{x \in X \mid R_i \leq d(x,x_0) < R_{i+1} \} = d_{x_0}^{-1}([R_i,R_{i+1})). \]
We can quantize $\mu$ with $N := k^d+1$ points by sending $X_0$ to a single point (which we will take to be $x_0$), and quantizing each $X_i$ by $k^{d-1}$ points (which will be a cover of $S(R_i)$). By Lemma \ref{lem:quantadd}, we have
\[ V_{k^d+1,p}(\mu) \leq V_{1,p}^{(\mu)}(X_0) + \sum_{i=1}^k V_{k^{d-1},p}^{(\mu)}(X_i). \]

We first cite the following simple lemma:

\begin{lemma}\label{lem:distbd}
Let $x \in X \setminus \{x_0\}$ and $0 \leq R \leq d(x,x_0)$. 
Then $d(x, S(R)) = d(x,x_0)-R$.
\end{lemma}
\begin{proof}
For any $y \in S(R)$, we have 
\[ d(x,x_0) \leq d(x,y) + d(y,x_0) = d(x,y) + R, \]
thus also $d(x, S(R)) \geq d(x,x_0)-R$.
Conversely, taking a constant-speed length-minimizing geodesic from $x_0$ to $x$, the geodesic intersects $S(R)$ at a unique point $z$, which satisfies
\[ d(x,x_0) = d(x,z) + d(z,x_0) = d(x,z) + R, \]
hence $d(x,S(R)) = d(x,z) = d(x,x_0)-R$.
\end{proof}

This is the only place where we use the assumption of $X$ being a geodesic space. 
With this expression, we can bound the quantization cost of each annulus:

\begin{lemma}\label{lem:quantann}
Let $0 \leq R_1 < R_2 \leq \infty$. Then for each $m \in \N$,
\[ V_{m,p}^{(\mu)}(A[R_1, R_2)) \leq \int_{R_1}^{R_2} \left[ (R-R_1) + r_{m}(S(R_1)) \right]^p \di \mu_1(R), \]
where by convention $\int_{R_1}^{R_2} = \int_{[R_1,R_2)}$.
\end{lemma}
\begin{proof}
Let $\delta > r_{m}(S(R_1))$, and let $S$ be a $\delta$-cover for $S(R)$ with at most $m$ elements. Then for each $x \in A[R_1, R_2)$,
\[ d(x, A) \leq \inf_{y \in S(R_1)} \left[ d(x,y) + d(y,S) \right] \leq \inf_{y \in S(R_1)} \left[ d(x,y) + \delta \right] = d(x, S(R_1)) + \delta. \]
Since $d(x, S(R_1)) = d(x,x_0) - R_1$ by Lemma \ref{lem:distbd}, we have
\begin{align*} 
V_{m,p}^{(\mu)}(A[R_1, R_2)) & \leq \int_{A[R_1, R_2)} d(x,S)^p \di\mu(x)
\\ & \leq \int_{A[R_1, R_2)} \left[ (d_{x_0}(x) - R_1) + \delta \right]^p \di \mu(x)
\\ & = \int_{R_1}^{R_2} \left[ (R - R_1) + \delta \right]^p \di \mu_1(R).
\end{align*}
Taking the infimum of the right-hand side over all $\delta > r_{m}(S(R_1))$ yields the statement.
\end{proof}

\begin{figure}[H]
\centering
\includegraphics[scale=0.75]{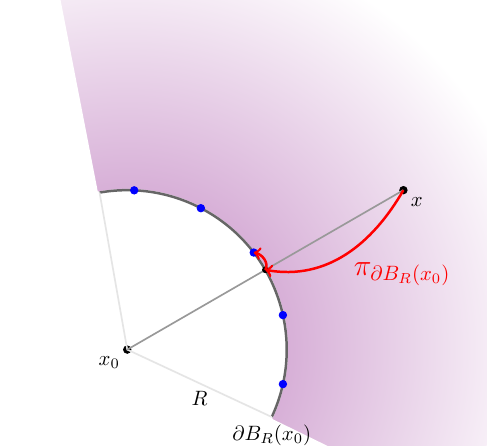}
\caption{Mass lying outside a sphere of radius $R$, first projected to the sphere, and then to the nearest element of a cover of the sphere.}
\end{figure}

Note that for $R_1 = 0$, $S(R_1)$ is the singleton $\{x_0\}$ and its covering radii are identically $0$.
Applying Lemma \ref{lem:quantann} to each annulus $X_i$, we then have
\begin{align*} 
V_{1,p}^{(\mu)}(X_0) & \leq \int_{0}^{R_1} R^p \di \mu_1(R); \\
V_{k^{d-1},p}^{(\mu)}(X_i) & \leq \int_{R_i}^{R_{i+1}} \left[ (R-R_i) + r_{k^{d-1}}(S(R_i)) \right]^p \di \mu_1(R), \quad 1 \leq i \leq N.
\end{align*}
Therefore, by the identity $(a+b)^p \leq 2^{p-1} (a^p + b^p)$ for $a, b \geq 0$,
\begin{align*} 
k^p V_{k^d+1,p}(\mu) & \leq 2^{p-1} \left[ k^p \sum_{i=0}^N \int_{R_i}^{R_{i+1}} (R-R_i)^p \di\mu_1(R) + \sum_{i=0}^N \int_{R_i}^{R_{i+1}} k^p r_{k^{d-1}}(S(R_i))^p \di \mu_1(R) \right]
\\ & \leq 2^{p-1} \left[ k^p \sum_{i=0}^N \int_{R_i}^{R_{i+1}} (R-R_i)^p \di\mu_1(R) + \int_0^\infty f(R)^p \di\mu_1(R) \right]
\end{align*}
where, since $f$ is nondecreasing, $N r_{N^{d-1}}(S(R_i)) \leq f(R_i) \leq f(R)$ for $R \in [R_i,R_{i+1})$.

Finally, by Lemma \ref{lem:piercefloor}, the radii $(R_i)_{i=1}^k$ can be chosen such that
\[ k^p \sum_{i=0}^k \int_{R_i}^{R_{i+1}} (R-R_i)^k \di\mu_1(R) \leq C \int_0^\infty (1+R^{p+\delta}) \di \mu_1(R) \]
for a constant $C = C(p,\delta)$ independent of $\mu$ and $k$. This proves the non-asymptotic upper bound. 

The upper bound on the limit supremum follows from the observation that, for any $N \geq 2$, taking $k = k(N) := \lfloor (N-1)^{1/d} \rfloor \geq 1$ such that $N \geq k^d+1$, we have
\[ N^{p/d} V_{N,p}(\mu) \leq N^{p/d} V_{k^d+1}(\mu) = \left(\frac{N}{k^d}\right)^{p/d} k^p V_{k^d+1}(\mu), \]
so taking the limit supremum and noting that $\lim_{N\to\infty} \frac{N}{k(N)^d} = 1$,
\begin{align*} 
\limsup_{N\to\infty} N^{p/d} V_{N,p}(\mu) & \leq \limsup_{k\to\infty} k^p V_{k^d+1,p}(\mu)
\\ & \leq C \limsup_{k\to\infty} \int_X \left(1+d(x, x_0)^{p+\delta}+f(d(x,x_0))^p \right) \di \mu(x). 
\end{align*}
This yields the last statement in Theorem \ref{thm:quantub}.
\end{proof}

\subsection{Proof of the asymptotic formula}\label{sect:proofasym}

\begin{proof}[Proof of Theorem \ref{thm:quantasym}]
The statement has already been proven by Kloeckner \cite{discapprox} and Iacobelli \cite[Thm. 1.4]{iacasym} for measures of compact support. Hence for any $K \subseteq M$ compact,
\[ Q_{p,d}^{(\mu)}(K) = \lim_{N \to \infty} N^{p/d} V_{N,p}^{(\mu)}(K) = Q_p([0,1]^d) \left( \int_K \rho(x)^{\frac{d}{d+p}} \di x \right)^{\frac{d+p}{d}}. \]
For each $\varepsilon > 0$, by Ulam's lemma there exists $K \subseteq M$ compact such that
\[ \int_{M \setminus K} \left[ 1 + d(x,x_0)^{p+\delta} + f(d(x,x_0))^p \right] \di \mu(x) < \varepsilon. \]
Consequently, by Theorem \ref{thm:quantub},
\[ \overline{Q}^{(\mu)}_{p,d}(M \setminus K)^p = \limsup_{N\to\infty} N^{p/d} V_{N,p}^{(\mu)}(M \setminus K) < C \varepsilon. \]
Lemma \ref{lem:quantcoeffadd} then implies that, for $r = \frac{p+d}{d}$,
\[ \overline{Q}_{p,d}(\mu)^r \leq {Q}^{(\mu)}_{p,d}(K)^r + (C \varepsilon)^{r/p}. \]
Taking the supremum of the right-hand side over all $K \subseteq M$ compact,
\[ \overline{Q}_{p,d}(\mu)^r \leq \sup_{K \subseteq M \text{ compact}} {Q}^{(\mu)}_{p,d}(K)^r + (C \varepsilon)^{r/p}. \]
Since this inequality holds for arbitrary $\varepsilon > 0$, we have
\[ \overline{Q}_{p,d}(\mu) \leq \sup_{K \subseteq M \text{ compact}} {Q}^{(\mu)}_{p,d}(K). \]
However, by the monotonicity of the quantization error, the right-hand side is also a lower bound on $\underline{Q}_{p,d}(\mu)$. This concludes that $Q_{p,d}(\mu)$ exists, and is given by
\begin{align*} 
Q_{p,d}(\mu) = \lim_{N \to \infty} N^{p/d} V_{N,p}(\mu) 
& = \sup_{K \subseteq M \text{ compact}} Q_{p,d}^{(\mu)}(K)
\\ & = Q_p([0,1]^d) \sup_{K \subseteq M \text{ compact}} \left( \int_K \rho(x)^{\frac{d}{d+p}} \di x \right)^{\frac{d+p}{d}}
\\ & = Q_p([0,1]^d) \left( \int_M \rho(x)^{\frac{d}{d+p}} \di x \right)^{\frac{d+p}{d}}. \qedhere
\end{align*}
\end{proof}

\section{Estimates on covering growth}\label{sect:covest}

We now present a few special cases in which the covering growth of manifolds can be estimated. We first formulate a general volumetric upper bound on covering numbers using Lemma \ref{lem:supregbd}, and then apply this bound on the specific settings of lower bounded Ricci curvature and geometric group actions.

\subsection{General volumetric estimates}\label{sect:volest}
In order to bound the expression $N(A; r) r^m$ over all $r > 0$ sufficiently small, we introduce the following definitions:

\begin{defin}
Let $X$ be a metric space, $\nu$ a Borel measure on $X$. 

For $A \subseteq X$ compact and $m, r_0 > 0$, define
\begin{align*} 
C_m(A; r_0) := \sup_{\substack{ 0 < r \leq r_0}} N(A; r) r^m; \quad
\vartheta_m(A; r_0) := \inf_{\substack{x \in A \\ 0 < r \leq r_0}} r^{-m} \nu(B_r(x)). 
\end{align*}
\end{defin}

The quantity $\vartheta_d(A; r_0)$ can be thought of as the lower density of $A$ with respect to the measure $\nu$, while $C_d(A; r_0)$ is a coarse quantification of the size of the set $A$. Indeed, as $r \to 0^+$, the quantity $N(A; r) r^m$ becomes comparable to the $m$-dimensional \emph{Minkowski content} of $A$; for more detail, refer to Appendix \ref{app:mincon}. We will use $C_m(A; r_0)$ to control the covering growth of spheres with $A = \partial B_R(x_0)$.

\begin{example}\label{ex:rdvartheta}
The $d$-dimensional Lebesgue measure $\L^d$ on $\R^d$ satisfies $\L^d(B_r(x)) = \omega_d r^d$ for all $x \in \R^d$ and $r > 0$, where $\omega_d := \L^d(B_1(0))$ is the volume of the unit $d$-dimensional ball. 

Consequently, for the choice $\nu = \L^d$ on $\R^d$, $\vartheta_d(A; r_0) = \omega_d$ independently of $A$ and $r_0$.
\end{example}

With these definitions, we can formulate the following general upper bound on covering numbers:

\begin{prop}\label{prop:supregbd}
Let $X$ be a metric space, $\nu$ a Borel measure on $X$, $A \subseteq X$ compact. For each $d \geq m > 0$ and $r_0 > 0$, we have
\begin{equation}\label{eq:supregbd}
 \sup_{\substack{N \in \N \\ r_N(A) \leq 2 r_0}} N r_N(A)^m \leq C_m(A; 2 r_0) \leq 2^m \vartheta_d(A; r_0)^{-1} \sup_{0 < r \leq r_0} \frac{\nu(A^r)}{r^{d-m}}. 
\end{equation}
If $\vartheta_d(A; r_0) = 0$, the right-hand side is set to $\infty$ and the bound is trivial.
\end{prop}
\begin{proof}
For the former inequality, observe that for each $N \in \N$ with $r_N(A) \leq 2 r_0$, $N < N(A; r_N(A))$ hence
\[ \sup_{\substack{N \in \N \\ r_N(A) \leq 2 r_0}} N r_N(A)^m \leq \sup_{\substack{N \in \N \\ r_N(A) \leq 2 r_0}} N(A; r_N(A)) r_N(A)^m \leq \sup_{0 < r \leq 2 r_0} N(A; r) r^m =: C_m(A; 2 r_0). \]
For the latter inequality, assume wlog $\vartheta_d(A; r_0) > 0$. 
Let $0 < r \leq 2 r_0$. For $\vartheta = \vartheta_d(A; r_0)$, the inequality $\nu(B_{r/2}(x)) \geq \vartheta (\frac{r}{2})^d$ holds for all $x \in A$, hence by Lemma \ref{lem:supregbd},
\[ N(A; r) \leq P\left(A; \frac{r}{2}\right) \leq \vartheta^{-1} \nu(A^{r/2}) \left(\frac{r}{2}\right)^{-d} = 2^m r^{-m} \vartheta^{-1} \frac{\nu(A^{r/2})}{(r/2)^{d-m}}. \]
Multiplying by $r^m$ and taking the supremum over all $r \in (0, 2 r_0]$ yields the latter inequality.
\end{proof}

We now restrict to geodesic spheres in order to bound the covering growth of manifolds or geodesic spaces.
Let $(M, g)$ be a complete $d$-dimensional Riemannian manifold and fix $x_0 \in M$. We apply Proposition \ref{prop:supregbd} to $A = \partial B_R(x_0)$ with $m = d - 1$ and $\nu = \vol_M$.

For $A = \partial B_R(x_0)$, observe that $A^{r} \subseteq A[R-r, R+r) = B_{R+r}(x_0) \setminus B_{R-r}(x_0)$, thus
\[ \vol_M(A^{r}) \leq \vol_M(B_{R+r}(x_0)) - \vol_M(B_{R-r}(x_0)). \]
For simplicity of notation, we define the following quantity:
\[ P_M(R; r) = \frac{\vol_M(B_{R+r}(x_0)) - \vol_M(B_{R-r}(x_0))}{2r}, \quad 0 < r \leq R < \infty. \]
This quantity is akin to an approximate Minkowski content for the sphere $\partial B_R(x_0)$, and will converge to $\cH^{d-1}(\partial B_R(x_0))$ as $r \to 0^+$ for almost every value of $R$. 

\begin{figure}[H]
\centering
\includegraphics[scale=0.7]{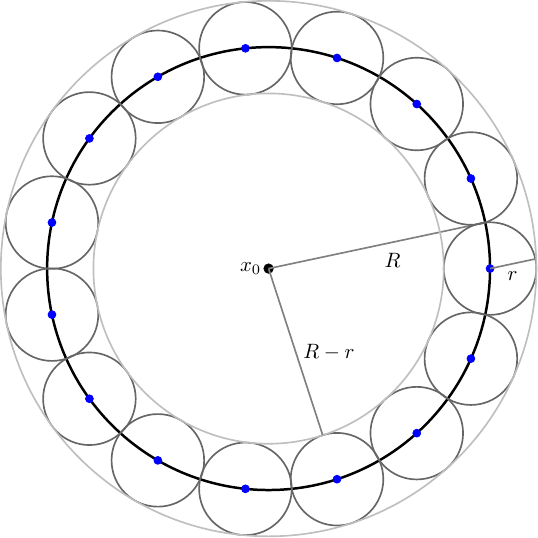}
\caption{An $r$-packing on the sphere $\partial B_R(x_0)$, contained in the annulus $B_{R+r}(x_0) \setminus B_{R-r}(x_0)$.}
\end{figure}

Therefore, by Proposition \ref{prop:supregbd}, for each $r_0 > 0$, we have
\begin{equation}\label{eq:covmfdub}
 \sup_{\substack{N \in \N \\ r_N(\partial B_R(x_0)) \leq 2 r_0}} N r_N(\partial B_R(x_0))^{d-1} \leq C_{d-1}(\partial B_R(x_0); 2 r_0) \leq \vartheta_d(\bar B_R(x_0); r_0)^{-1} \sup_{0 < r \leq r_0} P_M(R; r).
\end{equation}
In particular, for $r_0 = R/2$, $r_N(\partial B_R(x_0)) \leq R = 2 r_0$ thus
\begin{equation}
 \sup_{N \in \N} N r_N(\partial B_R(x_0))^{d-1} \leq C_{d-1}(\partial B_R(x_0); R) \leq \vartheta_d(\bar B_R(x_0); R/2)^{-1} \sup_{0 < r \leq R/2} P_M(R; r).
\end{equation}
Up to taking $(d-1)$th roots, the right-hand side thus provides a general upper bound for the covering growth of $M$.
We will estimate the two quantities on the right-hand side separately in the various special cases we will be considering in this section.

\begin{example}[Spheres in $\R^d$]
Take $\bS^{d-1}_R = \partial B_R(0) \subset \R^d$. For $\nu = \L^d$ and $m = d-1$, Proposition \ref{prop:supregbd} yields
\[ C_{d-1}(\bS^{d-1}_R; 2 r_0) \leq 2^{d-1} \vartheta_d(\bS^{d-1}_R; r_0)^{-1} \sup_{0 < r \leq r_0} \frac{\L^d((\bS^{d-1}_R)^r)}{r}. \]
We have $\vartheta_d(\bS^{d-1}_R; r_0) = \omega_d$ by Example \ref{ex:rdvartheta}, and since $(\bS^{d-1}_R)^r = B_{R+r}(0) \setminus \bar B_{R-r}(0)$,
\begin{align*} 
\sup_{0 < r \leq r_0} \frac{\L^d((\bS^{d-1}_R)^r)}{r} 
& = \sup_{0 < r \leq r_0} \frac{\omega_d (R+r)^d - \omega_d (R-r)^d}{r}
\\ & = 2 \omega_d \sup_{0 < r \leq r_0} \frac{1}{2r} \int_{R-r}^{R+r} d t^{d-1} \di t
\\ & \leq 2 \omega_d d (R+r_0)^{d-1},
\end{align*}
where in the last inequality we apply the convexity of $t \mapsto t^d$.
Therefore
\[ C_{d-1}(\bS^{d-1}_R; 2 r_0) \leq 2^d d (R+r_0)^{d-1}. \]
In particular, for all $r_0 \leq R/2$, the right-hand side is uniformly bounded by a constant multiple of $R^{d-1}$.
\end{example}

\begin{remark*}
Let $T \colon X \to Y$ be an $L$-Lipschitz map between metric spaces, $A \subseteq X$ compact.
Then since $N(T(A); L r) \leq N(A; r)$ for $T \colon X \to Y$ an $L$-Lipschitz map, given $r_0 > 0$ we have
\[ C_m(T(A); L r_0) = \sup_{\substack{ 0 < r \leq r_0}} N(T(A); L r) (L r)^m \leq L^m \sup_{\substack{ 0 < r \leq r_0}} N(A; r) r^m = L^m C_m(A; r_0). \]
\end{remark*}

\begin{example}[$\A_{x_0}$ bounds covering numbers]\label{ex:ax0covnum}
Let $(M, g)$ be a complete $d$-dimensional Riemannian manifold, $x_0 \in M$, and let $\A_{x_0}$ be defined as in Theorem \ref{thm:iacub}. Then $\exp_{x_0}$, restricted to a map from $\bS^{d-1}_R = \partial B_R(0) \subseteq T_v M$ to $\bar B_R(x_0) \subseteq M$, is $L$-Lipschitz with
\[ L = \sup_{v \in \bS^{d-1}_R} \|d_v \exp_{x_0}\| = \sup_{\substack{v \in \bS^{d-1}_R, \\ w \in T_v \bS^{d-1}_R, \\ \|w\|_v = 1}} \left\| d_v \exp_{x_0}[w] \right\|_{\exp_{x_0}(v)} = \frac{\A_{x_0}(R)}{R}. \]
Consequently,
\[ C_{d-1}(\partial B_R(x_0); 2 r_0) \leq L^{d-1} C_{d-1}\left(\bS^{d-1}_R; \frac{2 r_0}{L} \right). \]
Since $T_v M$, equipped with the norm induced by $g_v$, is isometric to $\R^d$, we obtain
\begin{align*}
C_{d-1}(\partial B_R(x_0); 2 r_0)
& \leq 2^d d L^{d-1} \left(R+\frac{2r_0}{L}\right)^{d-1} 
\\ & = 2^d d \left( \frac{\A_{x_0}(R)}{R} \right)^{d-1} R^{d-1} \left(1+\frac{2r_0}{\A_{x_0}(R)}\right)^{d-1}
\\ & = 2^d d \A_{x_0}(R)^{d-1} \left(1+\frac{2r_0}{\A_{x_0}(R)}\right)^{d-1}.
\end{align*}
In particular, since $L \geq 1$ (as long as $\partial B_R(x_0)$ contains antipodal points) we have $\A_{x_0}(R) \geq R$ thus 
\[ C_{d-1}(\partial B_R(x_0); 2 r_0) \leq 2^{2d-1} d \A_{x_0}(R)^{d-1} \quad \textup{for all } r_0 \leq R/2. \]
\end{example}

Thus $\A_{x_0}$ always yields a viable upper bound on the covering growth of Riemannian manifolds. The following special case illustrates the difference between the two notions of growth:

\begin{example}[Geodesic spheres for $d = 2$]\label{ex:twodim}
 Let $M$ be a $2$-dimensional complete Riemannian manifold, $x_0 \in M$. For a.e. $R > 0$, the sphere $\partial B_R(x_0)$ is a $1$-dimensional compact $\C^1$ submanifold of $M$ (e.g. by the implicit function theorem), and is thus the finite disjoint union of closed rectifiable curves. In this case, we can control the covering radii of $\partial B_R(x_0)$ in terms of its total length.
 
 For simplicity, assume $\partial B_R(x_0)$ is a closed rectifiable curve of length $L$, and take an arc length parametrization $s \colon [0,L] \to \partial B_R(x_0)$, which is a $1$-Lipschitz map. Note also that $r_k([0,L]) \leq \frac{L}{2k}$ (in fact, equality holds), covering $[0,L]$ by the $k$ points
 \[ x_i := \frac{2i-1}{2k} L \in [0,L], \quad i = 1, \ldots, k. \]
 Compare \cite[Ex. 5.5]{quantbook}. 
 Then for any $k \in \N$, we have
 \[ k r_k(\partial B_R(x_0)) = k r_k(s([0,L])) \leq k r_k([0,L]) = \frac{L}{2}. \]
The same upper bound will carry over when $\partial B_R(x_0)$ is a finite union of rectifiable curves. Denoting the total length (equiv. $1$-dimensional Hausdorff measure) of $\partial B_R(x_0)$ by $L_{x_0}(R)$, this implies that $M$ has $O(L_{x_0})$ covering growth.
\end{example}

In light of this observation, one can construct examples of $2$-manifolds where $L_{x_0}$ grows at a much slower rate than $\A_{x_0}$, e.g. when the derivative of the exponential map is large on a negligible sliver of each sphere. Indeed, we should expect $L_{x_0}$ to behave as the $L^1$ norm of $\di \exp_{x_0}$ over $\bS_{R} \subset T_{x_0} M$ instead of the $L^\infty$ norm. See Figure \ref{fig:parabfold} for an illustration of a surface where $L_{x_0}$ grows linearly but $\A_{x_0}$ at least exponentially with respect to the radius.

This correspondence between covering growth and the perimeters of spheres does not necessarily hold in higher dimensions, but it is always true that perimeters yield a \emph{lower bound} on the covering growth of manifolds; we illustrate this in Appendix \ref{app:mincon}.

\begin{figure}[H]\label{fig:parabfold}
\centering
\includegraphics[scale=0.75]{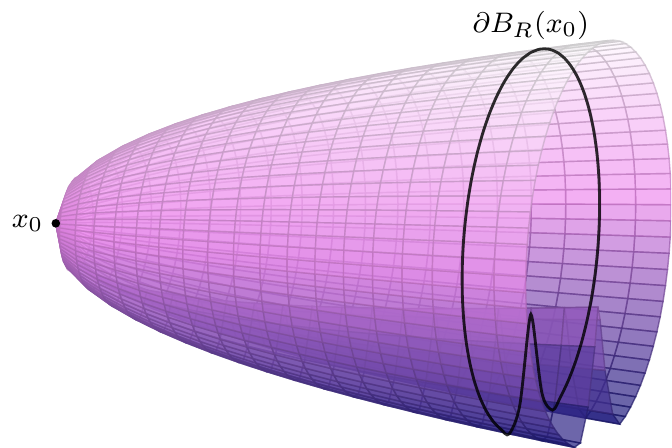}
\caption{
A parabolic $2$-dimensional surface on $\R^3$ with a fold along one direction. The surface is negatively curved along the fold, and the curvature can be made arbitrarily negative by sharpening the fold, causing $\di \exp_{x_0}$ to blow up along the same direction in the tangent space at $x_0$ and making $\A_{x_0}$ grow exponentially with $R$. However, the perimeters of spheres remain at most linearly proportional to the radius due to the overall parabolic growth of the surface.
}
\end{figure}

\subsection{Lower bounded Ricci curvature}\label{sect:riclb}

When the Ricci curvature of the manifold is bounded from below, both expressions in the right-hand side of \eqref{eq:covmfdub} can be controlled in terms of the volumes of balls. This allows us to show, in particular, that complete Riemannian manifolds with nonnegative Ricci curvature have $O(R)$ covering growth.

We first recall the \emph{Bishop-Gromov theorem}, which controls the volumes of balls under a lower bound on the Ricci curvature.

\begin{notation*}
We denote by $M^d_\kappa$ the $d$-dimensional complete, simply connected model Riemannian manifold with constant sectional curvature $\kappa \in \R$:
\[ M^d_\kappa = \left\{ \begin{matrix} \sqrt{\kappa} \bS^d, & \kappa > 0; \\ \R^d, & \kappa = 0; \\ \sqrt{-\kappa} \H^d, & \kappa < 0. \end{matrix} \right. \]
The diameter of $M^d_\kappa$ is
\[ D_\kappa = \diam(M^d_\kappa) = \left\{ \begin{matrix} \pi/\sqrt{\kappa}, & \kappa > 0; \\ \infty, & \kappa \leq 0. \end{matrix} \right. \]
The volume of a ball of radius $R \leq D_\kappa$ in $M^d_\kappa$ is given by
\[ \vol^d_\kappa(R) := d \omega_d \int_0^R \sin_\kappa^{d-1}(r) \di r, \]
where $\omega_d$ is the volume of the unit ball in $\R^d$, and
\[ \sin_\kappa(r) := \left\{ \begin{matrix} \frac{1}{\sqrt{\kappa}} \sin(\sqrt{\kappa} r), & \kappa > 0; \\ r, & \kappa = 0; \\ \frac{1}{\sqrt{-\kappa}} \sinh(\sqrt{-\kappa} r), & \kappa < 0. \end{matrix} \right. \]
The function $\sin_\kappa$ is exactly the solution to the \emph{Jacobi equation} $y^{\prime \prime}(r) + \kappa y(r) = 0$ with $y(0) = 0$ for constant sectional curvature $\kappa$.
\end{notation*}

\begin{theorem}[Bishop-Gromov theorem {\cite[\S 2.1]{Gromov1981}}]
Let $M$ be a complete $d$-dimensional Riemannian manifold with Ricci curvature bounded from below: $\mathrm{Ric} \geq (d-1)\kappa g$ for some $\kappa \in \R$. 

Then for any fixed $x \in M$, the function
\[ r \mapsto \frac{\vol_M(B_r(x))}{\vol^d_\kappa(r)} \]
is nonincreasing, and tends to $1$ as $r \to 0^+$. In other words, for $0 < r \leq R \leq D_\kappa$,
\[ \frac{\vol_M(B_r(x))}{\vol_M(B_R(x))} \geq \frac{\vol^d_\kappa(r)}{\vol^d_\kappa(R)} = \frac{\int_0^r \sin_\kappa(t) \di t}{\int_0^R \sin_\kappa(t) \di t}. \]
\end{theorem}

The Bishop-Gromov theorem thus provides simultaneous lower bounds on the volumes of small balls and upper bounds on those of large balls. This will allow us to estimate both quantities $\vartheta_d(\bar B_R(x_0); r_0)$ and $P_M(R; r)$ in \eqref{eq:covmfdub}.

\begin{lemma}[Bound on $\vartheta_d$] \label{lem:lamdbd}
Suppose $\mathrm{Ric} \geq (d-1) \kappa g$ on $M$, for some $\kappa \leq 0$. 

Fix $x_0 \in M$ and $R \geq r_0 > 0$. Then for any $r_1 \geq 0$,
\[ \vartheta_d(\bar B_R(x_0); r_0) \geq \omega_d \frac{\vol_M(B_{r_0+r_1}(x_0))}{\vol^d_\kappa({R+r_0+r_1})}. \]
\end{lemma}
\begin{proof}
Let $x \in \bar B_R(x_0)$, $0 < r \leq r_0$. Then by Bishop-Gromov,
\[ \frac{\vol_M(B_r(x))}{\vol^d_\kappa(r)} \geq \frac{\vol_M(B_{R+r_0+r_1}(x))}{\vol^d_\kappa({R+r_0+r_1})} \geq \frac{\vol_M(B_{r_0+r_1}(x_0))}{\vol^d_\kappa({R+r_0+r_1})}, \]
where $B_{r_0+r_1}(x_0) \subseteq B_{R+r_0+r_1}(x)$ since $d(x, x_0) = R$. Thus from the fact that $\vol^d_\kappa(r) \geq \omega_d r^d$ for $\kappa \leq 0$ (again by Bishop-Gromov), we obtain that
\[ \vol_M(B_r(x)) \geq \frac{\vol_M(B_{r_0+r_1}(x_0))}{\vol^d_\kappa({R+r_0+r_1})} \vol^d_\kappa(r) \geq \omega_d \frac{\vol_M(B_{r_0+r_1}(x_0))}{\vol^d_\kappa({R+r_0+r_1})} r^d. \qedhere \]
\end{proof}

\begin{figure}[H]
\centering
\includegraphics[scale=0.7]{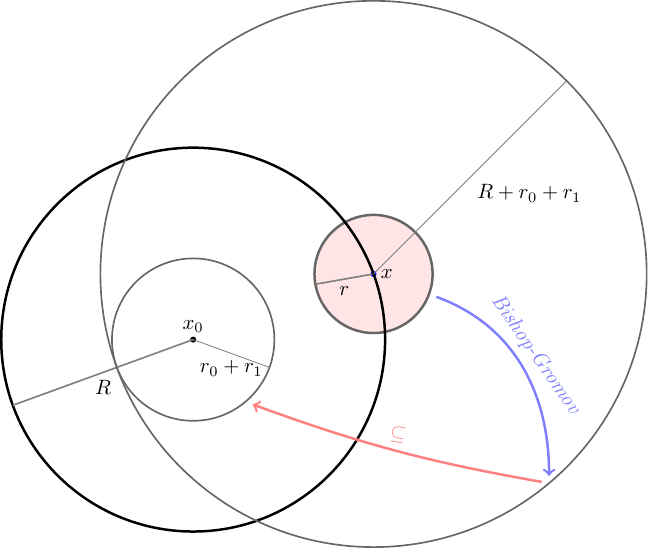}
\caption{Volume comparison argument for bound on $\vartheta_d$.}
\end{figure}

\begin{lemma}[Bound on $P_M$] \label{lem:perimbd}
Suppose $\mathrm{Ric} \geq (d-1)\kappa g$, $\kappa \leq 0$.

Fix $x_0 \in M$. Then for all $0 < r < R < \infty$, 
\[ P_M(R; r) \leq \frac{P^d_\kappa(R; r)}{\vol^d_\kappa({R+r})} \vol_M(B_{R+r}(x_0)), \]
where 
\[ P^d_\kappa(R; r) = \frac{\vol^d_\kappa({R+r}) - \vol^d_\kappa({R-r})}{2r} = d \omega_d \frac{1}{2r} \int_{R-r}^{R+r} \sin_\kappa^{d-1}(t) \di t \leq d \omega_d \sin_\kappa^{d-1}(R+r). \]
\end{lemma}
\begin{proof}
By Bishop-Gromov, we have
\begin{align*} 
P_M(R; r) 
& = \frac{1}{2r} \left[ \frac{\vol_M(B_{R+r}(x_0))}{\vol^d_\kappa({R+r})} \vol^d_\kappa({R+r}) - \frac{\vol_M(B_{R-r}(x_0))}{\vol^d_\kappa({R-r})} V^d_\kappa({R-r}) \right]
\\ & \leq \frac{1}{2r} \left[ \frac{\vol_M(B_{R+r}(x_0))}{\vol^d_\kappa({R+r})} \vol^d_\kappa({R+r}) - \frac{\vol_M(B_{R+r}(x_0))}{\vol^d_\kappa({R+r})} \vol^d_\kappa({R-r}) \right]
\\ & = \frac{\vol_M(B_{R+r}(x_0))}{\vol^d_\kappa({R+r})} P^d_\kappa(R; r). \qedhere
\end{align*}
\end{proof}

We bound the ratio $\frac{P^d_\kappa(R; r)}{V^d_\kappa({R+r})}$ in Lemma \ref{lem:pkbd}, from which we can deduce the following:

\begin{lemma}\label{lem:suppkbd}
There exist constants $C_1, C_2 > 0$ such that, for each $\kappa \leq 0$ and $R \geq 2 r_0 > 0$,
\[ \sup_{0 < r \leq r_0} \frac{P^d_\kappa(R; r)}{\vol^d_\kappa({R+r})} \leq \frac{C_1 + C_2 \sqrt{-\kappa}(R+r_0)}{R}. \]
\end{lemma}

In particular, as $R \to \infty$, the right-hand side will approach a finite positive value for $\kappa < 0$ and will converge to $0$ for $\kappa = 0$. We note also that for $\kappa = 0$, the constant $C_1$ can be taken to equal $d$.

Putting the above bounds together, we obtain

\begin{prop}\label{prop:ricgeqk}
Let $M$ be a $d$-dimensional complete Riemannian manifold with $\mathrm{Ric} \geq (d-1) \kappa g$, $\kappa \leq 0$. 

Let $x_0 \in M$, $R \geq 2 r_0 > 0$. Then for any $N \in \N$ such that $r_N(\partial B_R(x_0)) \leq 2 r_0$,
\[ N r_N(\partial B_R(x_0))^{d-1} \leq C_{d-1}(\partial B_R(x_0); 2r_0) \leq C \frac{C_1 + C_2 \sqrt{-\kappa}(R+r_0)}{R} \vol^d_\kappa({2R+r_0}), \]
where $C, C_1, C_2 > 0$ are constants independent of $x_0$, $R$ and $r_0$.

In particular, taking $r_0 = R/2$, we have that
\[ N r_N(\partial B_R(x_0))^{d-1} \leq C \frac{C_1 + \frac{3 C_2}{2} \sqrt{-\kappa} R}{R} \vol^d_\kappa\left(\frac{5R}{2}\right), \quad \textup{for all } N \in \N. \]
\end{prop} 
\begin{proof}
Let $N \in \N$ such that $r_N(\partial B_R(x_0)) \leq 2 r_0$. 
Applying Proposition \ref{prop:supregbd} with $\vartheta = \vartheta_d(\partial B_R(x_0); r_0)$, we have
\[ N r_N(\partial B_R(x_0))^{d-1} \leq C_{d-1}(\partial B_R(x_0); 2r_0) \leq 2^d \vartheta^{-1} \sup_{0 < r \leq r_0} P_M(R; r). \]
By Lemmas \ref{lem:perimbd} and \ref{lem:suppkbd}, we have
\begin{align*}
\sup_{0 < r \leq r_0} P_M(R; r) \leq \sup_{0 < r \leq r_0} \frac{P^d_\kappa(R; r)}{\vol^d_\kappa({R+r})} \vol_M(B_{R+r}(x_0)) \leq \vol_M(B_{R+r_0}(x_0)) \frac{C_1 + C_2 \sqrt{-\kappa}(R+r_0)}{R}.
\end{align*}
Applying Lemma \ref{lem:lamdbd} with the choice $r_1 = R$ also yields
\[ \vartheta^{-1} \leq \omega_d^{-1} \frac{\vol^d_\kappa({2R+r_0})}{\vol_M(B_{R+r_0}(x_0))}. \]
We therefore have
\[ N r_N(\partial B_R(x_0))^{d-1} \leq C_{d-1}(\partial B_R(x_0); 2r_0) \leq 2^d \omega_d^{-1} \vol^d_\kappa({2R+r_0}) \frac{C_1 + C_2 \sqrt{-\kappa}(R+r_0)}{R}. \qedhere \]
\end{proof}

In particular, we deduce that every complete connected Riemannian manifold with nonnegative Ricci curvature has $O(R)$ covering growth:

\begin{corol}\label{cor:nonnegric}
Let $M$ be a $d$-dimensional complete connected Riemannian manifold with nonnegative Ricci curvature. 
Then $M$ has uniform $O(R)$ covering growth. 

More specifically, for each $x_0 \in M$, $N \in \N$ and $R > 0$, the following bound holds:
\[ N^{\frac{1}{d-1}} r_N(\partial B_R(x_0)) \leq (5^d d)^{\frac{1}{d-1}} R. \]
\end{corol}
\begin{proof}
Let $x_0 \in M$, $R > 0$.
Then applying Proposition \ref{prop:ricgeqk} with $\kappa = 0$ and $r_0 := R/2 \geq r_1(\partial B_R(x_0))/2$, for each $N \in \N$, we have
\begin{align*}
N r_N(\partial B_R(x_0))^{d-1} 
& \leq 2^d \omega_d^{-1} \vol^d_0(B_{5R/2}) \sup_{0 < r \leq R/2} \frac{d}{R+r} 
\\ & = 2^d \omega_d^{-1} \left( \omega_d (5R/2)^d \right) \frac{d}{R} = 5^d d R^{d-1}. \qedhere
\end{align*}
\end{proof}

The constant on the right-hand side can be made more precise when $\vartheta_d$ admits a universal lower bound, in which case Lemma \ref{lem:lamdbd} need not be applied. This can be thought of as an assumption of \emph{at least Euclidean volume growth} for balls on $M$.

\begin{prop}\label{prop:ricgeqklam}
Let $M$ be a $d$-dimensional complete Riemannian manifold with $\mathrm{Ric} \geq (d-1) \kappa g$, $\kappa \leq 0$. 
Assume moreover that there exists $\vartheta > 0$ such that $\vol_M(B_r(x)) \geq \vartheta r^d$ for all $r > 0$.

Let $x_0 \in M$, $R \geq 2 r_0 > 0$. Then for any $N \in \N$ such that $r_N(\partial B_R(x_0)) \leq 2 r_0$,
\[ N r_N(\partial B_R(x_0))^{d-1} \leq C_{d-1}(\partial B_R(x_0); 2r_0) \leq 2^d \vartheta^{-1} \frac{C_1 + C_2 \sqrt{-\kappa}(R+r_0)}{R} \vol_M(B_{R+r_0}(x_0)), \]
where $C_1, C_2 > 0$ are as in Lemma \ref{lem:pkbd}.
\end{prop}

\begin{corol}\label{cor:nonnegriclam}
Let $M$ be a $d$-dimensional complete Riemannian manifold with nonnegative Ricci curvature. 
Assume moreover that there exists $\vartheta > 0$ such that $\vol_M(B_r(x)) \geq \vartheta r^d$ for all $r > 0$.

Then for each $x_0 \in M$, $N \in \N$ and $R > 0$, the following bound holds:
\[ N^{\frac{1}{d-1}} r_N(\partial B_R(x_0)) \leq \left( \frac{3^d d \omega_d}{\vartheta} \right)^{\frac{1}{d-1}} R. \]
\end{corol}
\begin{proof}
Set $r_0 := R/2$. Then for each $N \in \N$ we have $r_N(\partial B_R(x_0)) \leq R = 2 r_0$, thus
\[ N r_N(\partial B_R(x_0))^{d-1} \leq C_{d-1}(\partial B_R(x_0); 2r_0) \leq 2^d \vartheta^{-1} \frac{d}{R} \vol_M(B_{R+r_0}(x_0)). \]
By the assumption of nonnegative Ricci curvature, $\vol_M(B_{R+r_0}(x_0)) = \vol_M(B_{3R/2}(x_0)) \leq \omega_d (3R/2)^d$, thus
\[ N r_N(\partial B_R(x_0))^{d-1} \leq 3^d \vartheta^{-1} d \omega_d R^{d-1}. \qedhere \]
\end{proof}

The constants on the right-hand side may be reduced further if $N$ is taken to be larger, hence $r_0$ is taken to be close to $0$, but the bounds given above are restricted by the fact that they hold for all $N \in \N$.

For example, on $M = \R^d$, $\partial B_R(0) = \bS^{d-1}_R$ intersects the axes at $2^d$ points, which are separated by at least a distance of $\sqrt{2} R$. This implies that
\[ N(\partial B_R(0); \sqrt{2} R) \geq P(\partial B_R(0); \sqrt{2} R) \geq 2^d, \]
thus for $N = 2^d - 1$, $r_N(\partial B_R(0)) \geq \sqrt{2} R$ and
\[ N r_N(\partial B_R(0))^{d-1} \geq (2^d-1) 2^{\frac{d-1}{2}} R^{d-1}. \]

\subsection{Geometric group actions}\label{sect:cocompact}

In order to further demonstrate the applicability of the covering growth upper bound, we consider Riemannian manifolds subject to the \emph{geometric action} of a discrete group of isometries. 
In particular, when the group of isometries is of polynomial growth, as in the case of $\R^d$ equipped with a periodic metric, we show that a polynomial moment condition is sufficient even though the lower bound on the Ricci curvature might be arbitrarily negative. For reference on geometric group theory, we refer to Bridson and Haefliger \cite{Bridson}.

\begin{defin}
A group action $\Gamma \times X \to X$ of a group $\Gamma$ on a locally compact metric space $X$ is said to be a \emph{geometric} group action if
\begin{itemize}
\item each $\gamma \in \Gamma$ acts on $X$ by an isometry;
\item (\emph{proper discontinuity}) for each $K \subseteq X$ compact, the set $\{\gamma \in \Gamma \mid \gamma \cdot K \cap K \neq \varnothing \}$ is finite;
\item (\emph{cocompactness}) there exists $K \subseteq X$ compact such that $\Gamma \cdot K := \bigcup_{\gamma \in \Gamma} \gamma \cdot K = X$.
\end{itemize}
\end{defin}

\begin{figure}[H]
\centering
\includegraphics[scale=0.75]{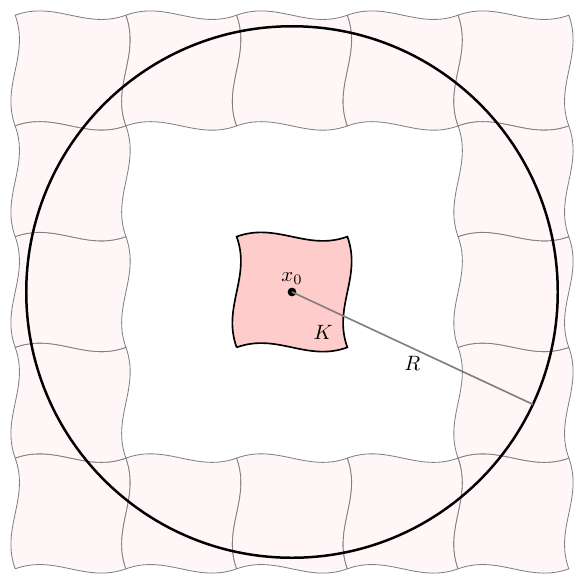}
\caption{A sphere covered by translates of a compact set $K$ along a geometric group action.}
\end{figure}

Let $M$ be a complete, connected $d$-dimensional Riemannian manifold, on which a discrete group $\Gamma$ acts geometrically. We seek to control the covering numbers of spheres in terms of the growth of $\Gamma$. In particular, we will prove the following statement:

\begin{prop}\label{prop:gcovbd}
Suppose $\Gamma$ acts geometrically on a complete $d$-dimensional Riemannian manifold $M$. Assume further that $\Gamma$ is of \emph{polynomial growth}: for any choice of generating set $\A \subset \Gamma$, there exist constants $C, \alpha > 0$ such that the \emph{growth function} $\beta_\A(k) \leq C k^\alpha$ for $k \in \N$ large. 

Then fixing $x_0 \in M$, for $R > 0$ sufficiently large and $N \in \N$ arbitrary, the following bound holds for the covering growth of spheres in $M$:
\[ N r_N(\partial B_R(x_0))^{d-1} \leq C_{d-1}(\partial B_R(x_0)) \leq C R^{\alpha+d-1}. \]
Thus $M$ has $O(R^{\frac{\alpha}{d-1}+1})$ covering growth around any point $x_0 \in M$.
\end{prop}

This bound more properly illustrates the coarse or large-scale nature of covering growth, as opposed to the growth of the exponential map. As with the material in the previous section, this estimate can also be extended directly to non-smooth metric measure spaces. Similar estimates can also be obtained for other bounds on the growth of the group, such as sub-exponential growth, using the same arguments given below.

\begin{example}[Periodic metrics on $\R^d$]\label{ex:permet}
Let $M = (\R^d, g)$, where $g$ is a $1$-periodic metric tensor on $\R^d$. That is, for any $a \in \Z^d$, $g_{x+a} = g_x$ (more rigorously, the metric tensor $g$ is preserved under pullback along each translation $x \mapsto x+a$).

Then the abelian group $\Z^d$ acts geometrically on $M$ by translations $x \mapsto x+a$, with $K = [0,1]^d$ compact in $M$ such that $\Z^d \cdot K = M$, and is of polynomial growth of order $d$. Thus any such manifold with a periodic metric tensor, however negatively curved they may be at points, will still have at most $O(R^{\frac{d}{d-1}+1})$ covering growth.

A simple example of a manifold of this type is a sinusoidal surface as illustrated in Figure \ref{fig:sinusoid}, given by the graph of a function of the form $f(x,y) = A \sin(\omega_x x + \phi_x) \sin(\omega_y y + \phi_y)$. Observe the saddle points occurring periodically on the surface. The larger the amplitude $A$ of the sinusoid, the more negatively curved the surface will be at the saddle points. Nevertheless, Proposition \ref{prop:gcovbd} guarantees polynomial covering growth independently of the curvature lower bound.
\end{example}

\begin{figure}[H]\label{fig:sinusoid}
\centering
\includegraphics[scale=0.75]{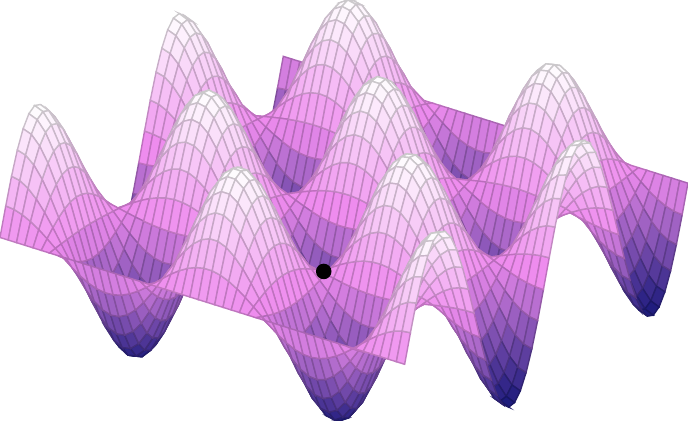}
\caption{
Sinusoidal surface in $\mathbb{R}^3$, with one saddle point highlighted in black.
}
\end{figure}

We now proceed with the proof of Proposition \ref{prop:gcovbd}.
We first recall some basic background on the growth of groups:

\begin{defin}
Let $\Gamma$ be a finitely generated group. Given a finite generating set $\A \subset \Gamma$, the \emph{word length} of $\gamma \in \Gamma$ is the length of the shortest representation of $\gamma$ as a product of elements of $\A \cup \A^{-1}$:
\[ |\gamma|_\A := \min \{ k \in \N \mid \exists \alpha_0 = e, \alpha_1, \ldots, \alpha_k \in \A \cup \A^{-1}, \gamma = \alpha_0 \alpha_1 \ldots \alpha_k \} \]
where we set $|e|_\A := 0$ for the identity element $e \in \Gamma$. The \emph{word metric} induced by $\A$ is given by
\[ d_\A(\gamma, \eta) := |\gamma^{-1} \eta|_\A, \quad \gamma, \eta \in \Gamma. \]
The \emph{growth function} of $\Gamma$ induced by $\A$ measures the cardinality of closed balls centered at $e$ with respect to the word metric:
\[ \beta_\A(k) := \#\{\gamma \in \Gamma \mid |\gamma|_\A \leq k\} = \#\left[ (\A \cup \A^{-1} \cup \{e\})^k\right]. \]
If there exist constants $C, \alpha, k_0 > 0$ such that $\beta_\A(k) \leq C k^\alpha$ for all $k \geq k_0$, we say that $\Gamma$ is of \emph{polynomial growth} of order $\alpha$.
\end{defin}

The word metric indeed defines a metric on $\Gamma$, and any two word metrics are bi-Lipschitz equivalent: given finite generating sets $\A, \A^\prime \subset \Gamma$, setting $\lambda := \max_{\alpha \in \A} |\alpha|_{\A^\prime}$, we have
\[ d_{\A^\prime}(\gamma,\eta) \leq \lambda d_\A(\gamma, \eta), \quad \textup{for all } \gamma, \eta \in \Gamma. \]
In particular, $|\gamma|_{\A^\prime} \leq \lambda |\gamma|_{\A}$ for all $\gamma \in \Gamma$, hence $\beta_\A(k) \leq \beta_{\A^\prime}(\lambda k)$ for all $k \in \N$. The property of polynomial growth is thus independent of the choice of generating set $\A$.

\begin{example}
The word metric on $\Z^d$ with respect to the canonical generators $\{e_i\}_{i=1}^d$ is exactly the metric induced by the $\ell^1$ norm. 
\end{example}

We also recall the definition of \emph{quasi-isometries} between metric spaces:

\begin{defin}
Let $X$, $Y$ be metric spaces. A (not necessarily continuous) map $\phi \colon X \to Y$ is called a \emph{$(\lambda, \eps)$-quasi-isometric embedding} for $\lambda \geq 1$ and $\eps \geq 0$ if the following inequalities hold for all $x, x^\prime \in X$:
\[ \lambda^{-1} d_X(x,x^\prime) - \eps \leq d_Y(\phi(x),\phi(x^\prime)) \leq \lambda d_X(x,x^\prime) + \eps. \]
If moreover $\sup_{y \in Y} d(y, \phi(X)) < \infty$, $\phi$ is called a \emph{$(\lambda,\eps)$-quasi-isometry} and $X$ and $Y$ are said to be \emph{quasi-isometric}.
\end{defin}

Quasi-isometries can be said to capture the coarse or large-scale geometry of metric spaces, as opposed to the local structure.
The following fundamental result provides a correspondence between the growth of groups and the metric spaces they act on:

\begin{prop}[Schwarz-Milnor Lemma]\label{prop:milsch}
Suppose a group $\Gamma$ acts geometrically on a locally compact length space $X$. Then $\Gamma$ is finitely generated, and for any choice of word metric $d_\A$ on $\Gamma$ and base point $x_0 \in X$, the map $\phi \colon \Gamma \to X$ given by $\phi(\gamma) := \gamma \cdot x_0$ is a quasi-isometry.
\end{prop}

For a proof and further background, we refer to Bridson and Haefliger \cite[Prop. 8.19]{Bridson}. We use this result to control the sizes of balls and annuli around $x_0$:

\begin{notation*}
We denote by $N_\Gamma(A; B)$ the minimum size of a cover of a set $A \subseteq X$ by images of another set $B \subseteq X$ along the action of $\Gamma$:
\[ N_\Gamma(A; B) := \min\{ |\G| \mid \G \subseteq \Gamma, A \subseteq \G \cdot B \}. \]
\end{notation*}

\begin{lemma}\label{lem:gcovnum}
Suppose $\Gamma$ acts geometrically on a proper length space $X$. Fix $x_0 \in X$ and a generating set $\A \subset \Gamma$. Then there exist $\lambda \geq 1$, $\eps \geq 0$, $\delta > 0$ such that for any $R > 0$ and $r_0 > \delta$, 
\[ N_\Gamma( \bar B_R(x_0); \bar B_{r_0}(x_0) ) \leq \beta_\A\left( \lfloor \lambda (R + r_0 + \eps) \rfloor \right). \]
Moreover, for any $R_2 > R_1 \geq r_0 + \eps$, 
\[ N_\Gamma( \bar B_{R_2}(x_0) \setminus \bar B_{R_1}(x_0); \bar B_{r_0}(x_0) ) \leq \beta_\A\left( \lfloor \lambda (R_2 + r_0 + \eps) \rfloor \right) - \beta_\A\left( \lfloor \lambda^{-1} (R_1 - r_0 - \eps) \rfloor \right). \]
\end{lemma}
\begin{proof}
Let $\phi \colon \Gamma \to X$ be the quasi-isometry given by the Schwarz-Milnor lemma, with constants $\lambda \geq 1$ and $\eps \geq 0$. Set moreover $\delta := \sup_{x \in X} d(x, \phi(\Gamma)) < \infty$. 

Let $r_0 > \delta$.
Then the closed ball $B_0 = \bar B_{r_0}(x_0)$ is compact and its images $\gamma \cdot B_0 = \bar B_{r_0}(\gamma \cdot x_0)$ (noting that $\Gamma$ acts by isometries) cover $X$, since for any $x \in X$, there exists $\gamma \in \Gamma$ such that $d(x, \gamma \cdot x_0) < \frac{r_0}{\delta} d(x, \phi(\Gamma)) \leq r_0$. 

Consequently, for any (relatively) compact subset $A \subseteq X$, $N_\Gamma(A; B_0)$ is finite and we have the trivial bound
\[ N_\Gamma(A; B_0) \leq \#\{\gamma \in \Gamma \mid (\gamma \cdot B_0) \cap A \neq \emptyset\}. \]
We bound the right-hand side separately for $A = \bar B_R(x_0)$ and $A = \bar B_{R_2}(x_0) \setminus \bar B_{R_1}(x_0)$.

Let $R > 0$, and let $\gamma \in \Gamma$ such that there exists $x \in \bar B_R(x_0) \cap \gamma \cdot B_0$. Since $\phi$ is a $(\lambda,\eps)$-quasi-isometry,
\[ \lambda^{-1} d_\A(\gamma, e) - \eps \leq d(\gamma \cdot x_0, x_0) \leq d(\gamma \cdot x_0, x) + d(x, x_0) \leq r_0 + R. \]
Thus we must have
\[ |\gamma|_\A = d_\A(e,\gamma) \leq \lambda(R + r_0 + \eps). \]
Noting that $|\gamma|_\A$ is always integer-valued, setting $k := \lfloor \lambda(R + r_0 + \eps) \rfloor$ we then obtain
\[ \#\{\gamma \in \Gamma \mid (\gamma \cdot B_0) \cap \bar B_{R}(x_0) \neq \emptyset\} \leq \#\{\gamma \in \Gamma \mid |\gamma|_\A \leq k \} = \beta_\A(k). \]
This yields the first inequality.

For the second inequality, observe that
\begin{align}\label{eq:annulcard}
\#\{\gamma \in \Gamma \mid (\gamma \cdot B_0) \cap (\bar B_{R_2}(x_0) \setminus \bar B_{R_1}(x_0)) \neq \emptyset\}
\\ = \#\{\gamma \in \Gamma \mid (\gamma \cdot B_0) \cap \bar B_{R_2}(x_0) \neq \emptyset\}
  - \#\{\gamma \in \Gamma \mid (\gamma \cdot B_0) \subseteq \bar B_{R_1}(x_0)\}. 
\end{align}
The first term can be upper bounded as above. Now set $k_1 := \lfloor \lambda^{-1} (R_1 - r_0 - \eps) \rfloor$ and take $\gamma \in \Gamma$ such that $|\gamma|_\A \leq k_1$. Given any $x \in \gamma \cdot B_0 = \bar B_{r_0}(\gamma \cdot x_0)$, again by the quasi-isometricity of $\phi$,
\[ d(\gamma \cdot x_0, x_0) \leq \lambda |\gamma|_\A + \eps \leq R_1 - r_0, \]
and by the triangle inequality,
\[ d(x, x_0) \leq d(x, \gamma \cdot x_0) + d(\gamma \cdot x_0, x_0) \leq R_1. \]
Therefore $x \in \bar B_{R_1}(x_0)$. This shows that $|\gamma|_\A \leq k_1$ implies $\gamma \cdot B_0 \subseteq \bar B_{R_1}(x_0)$, in particular
\[ \#\{\gamma \in \Gamma \mid (\gamma \cdot B_0) \subseteq \bar B_{R_1}(x_0)\} \geq \#\{\gamma \in \Gamma \mid |\gamma|_\A \leq k_1 \} = \beta_\A(k_1). \]
Applying this to \eqref{eq:annulcard} yields the second inequality.
\end{proof}

We now apply this bound, in combination with \eqref{eq:covmfdub}, in order to control the covering growth of geodesic spheres on Riemannian manifolds. Henceforth we assume $M$ to be a $d$-dimensional complete Riemannian manifold on which a group $G$ acts geometrically. 

Fixing $x_0 \in M$ and $r_0 > \delta$, we bound $N(\partial B_R(x_0); r) r^{d-1}$ by considering the cases $r < r_0$ and $r \geq r_0$ separately.

For $r < r_0$, we again apply Proposition \ref{prop:supregbd}. The estimate given in Lemma \ref{lem:gcovnum} is too imprecise to bound the approximate perimeters of balls directly, so we will instead proceed as in the previous section by bounding $P_M(R; r) \leq C \vol_M(B_{R+r}(x_0))$.

\begin{lemma}\label{lem:glamd}
Fix $r_0 > 0$.
There exists $\lambda > 0$ dependent on $r_0$ such that
\[ \vol_M(B_r(x)) \geq \lambda r^d , \quad \textup{for all } x \in M, r \leq r_0. \]
Consequently, $\lambda_d(\bar B_R(x_0); r_0) \geq \lambda$ independently of $R > 0$.
\end{lemma}
\begin{proof}
Take $K \subseteq M$ compact such that $\Gamma \cdot K = M$. By Lemma \ref{lem:limdens}, there exists $\lambda > 0$ such that
\[ \vol_M(B_r(x)) \geq \lambda r^d , \quad \textup{for all } x \in K, r \leq r_0. \]

Now suppose $x \in M$ and $r \leq r_0$. Take $\gamma \in \Gamma$ such that $\gamma \cdot x \in K$, hence
\[ \vol_M(B_r(x)) = \vol_M(\gamma \cdot B_r(x)) = \vol_M(B_r(\gamma \cdot x)) \geq \lambda r^d \]
since $\Gamma$ acts on $M$ by isometries.
\end{proof}

\begin{lemma}\label{lem:gperim}
There exists a constant $C$ dependent on $M$, $x_0$ and $r_0$ such that, for any $0 < r \leq r_0 \leq R/2 < \infty$,
\[ P_M(R; r) \leq C \vol_M(B_{R+r_0}(x_0)). \]
\end{lemma}
\begin{proof}
Take $K \subseteq M$ compact such that $\Gamma \cdot K = M$. Since the Ricci curvature tensor is bounded on $K$, and preserved under the isometric action of each $\gamma \in \Gamma$, $M$ has bounded Ricci curvature. The result then follows from Lemmas \ref{lem:perimbd} and \ref{lem:suppkbd}.
\end{proof}

Again if $M$ has nonnegative Ricci curvature, one can further divide the right-hand side by $R$, but we ignore this distinction here since the case of nonnegative Ricci curvature has already been treated in the previous section.

\begin{lemma}[$r < r_0$]\label{lem:rltr0}
Let $x_0 \in M$, $\lambda$, $\eps$ and $r_0$ as above. Suppose $\Gamma$ is of polynomial growth with exponent $\alpha > 0$. Then there exists $C > 0$ such that for $R > r_0$,
\[ C_{d-1}(\partial B_R(x_0); r_0) = \sup_{0 < r \leq r_0} N(\partial B_R(x_0); r) r^{d-1} \leq C R^{\alpha}. \]
\end{lemma}
\begin{proof}
Proposition \ref{prop:supregbd} and Lemmas \ref{lem:glamd} and \ref{lem:gperim} yield
\[ C_{d-1}(\partial B_R(x_0); r_0) \leq C \vol_M(B_{R+r_0}(x_0)). \]
Then since $\Gamma$ acts by isometries and in particular preserves volume, we can apply the first inequality in Lemma \ref{lem:gcovnum} to obtain
\begin{align*} 
\vol_M(B_{R+r_0}(x_0)) & \leq \vol_M(B_{r_0}(x_0)) N_\Gamma(B_{R+r_0}(x_0); B_{r_0}(x_0))
\\ & \leq \vol_M(B_{r_0}(x_0)) \beta_\A\left( \lfloor \lambda (R + r_0 + \eps) \rfloor \right)
\\ & \leq C^\prime (R+r_0+\eps)^\alpha
\end{align*}
for some constant $C^\prime > 0$ depending on $x_0$, $r_0$, $\eps$ and the growth rate of $\beta_\A$, but not on $R$. Note lastly that for $R > r_0$, $\frac{R+r_0+\eps}{R} < \frac{r_0+\eps}{r_0}$ which is likewise a constant.
\end{proof}

For $r \geq r_0$, we can apply the second inequality in Lemma \ref{lem:gcovnum}, but it does not affect the asymptotic behavior of the bound obtained unless $\lambda = 1$ and $\beta_\A$ is strictly polynomial in the sense that $\frac{\beta_\A(k)}{k^\alpha}$ converges to some $C \in (0,\infty)$. We nevertheless include the general bound before restricting to polynomial growth:

\begin{lemma}[$r \geq r_0$]\label{lem:rgeqr0}
Let $x_0 \in M$, $\A$, $\lambda$, $\eps$ and $r_0$ as above. 
Given $R > r_0$, the following bound holds:
\[ \sup_{r_0 \leq r \leq R} N(\partial B_R(x_0); r) r^{d-1} \leq C R^{d-1} \left[ \beta_\A\left( \lfloor \lambda (R + r_0 + \eps) \rfloor \right) - \beta_\A\left( \lfloor \lambda^{-1} (R_1 - r_0 - \eps) \rfloor \right) \right] \]
In particular, if $\Gamma$ is of polynomial growth of order $\alpha > 0$,
\[ \sup_{r_0 \leq r \leq R} N(\partial B_R(x_0); r) r^{d-1} \leq C R^{\alpha+d-1}. \]
\end{lemma}
\begin{proof}
The first inequality follows simply by
\[ N(\partial B_R(x_0); r) r^{d-1} \leq N(\partial B_R(x_0); r_0) R^{d-1} \leq N_G(\bar B_R(x_0) \setminus \bar B_{R-r_0}(x_0); \bar B_{r_0}(x_0)) R^{d-1}. \]
The second inequality then follows by substituting $\beta_\A(k) \leq C k^\alpha$.
\end{proof}

\begin{remark*}
One could also partition into cases $k r_0 \leq r \leq (k+1) r_0$, and bound
\[ \sup_{k r_0 \leq r \leq (k+1) r_0} N(\partial B_R(x_0); r) r^{d-1} \leq N(\partial B_R(x_0); k r_0) \left((k+1) r_0\right)^{d-1}, \]
but this does not improve the exponent $R^{\alpha+d-1}$: Lemma \ref{lem:gcovnum} bounds the first term on the right-hand side by an expression on the order of $(R+k r_0)^\alpha$, and
\[ \sup_{\substack{k \in \N \\ k r_0 \leq R}} (R+k r_0)^\alpha (k+1)^{d-1} \geq C R^{\alpha+d-1}, \]
as can be observed by taking $k \in \N$ such that $R \approx 2k r_0$.
\end{remark*}

These two lemmas together prove Proposition \ref{prop:gcovbd}.

\begin{remark*}
The exponent $\frac{\alpha}{d-1}+1$ in Proposition \ref{prop:gcovbd} is not sharp in general; this is because the notion of quasi-isometry is still too coarse to capture the sizes of spheres precisely. The $+1$ term, which arises only in the $r \geq r_0$ case, could be eliminated if $\beta_\A$ is strictly polynomial and the quasi-isometry constant $\lambda$ associated to $\A$ is equal to $1$.
However, the term $\frac{\alpha}{d-1}$ that also appears in Lemma \ref{lem:rltr0} cannot be improved further using word metrics.

To illustrate this, consider the abelian group $\Z^d$ acting geometrically on $\R^d$ by translations, as a special case of Example \ref{ex:permet}. $\Z^d$ is of polynomial growth with exponent $\alpha = d$, so Proposition \ref{prop:gcovbd} implies $O(R^{\frac{d}{d-1}+1})$ covering growth, but $\R^d$ is simply of $O(R)$ covering growth. 

The inaccuracy here arises from Lemma \ref{lem:gcovnum}: while we expect annuli of the form $A := B_{R+r_0}(x_0) \setminus B_{R-r_0}(x_0)$ to be covered by $C R^{d-1}$ cubes for $R \gg r_0$, the upper bound on $N_{\Z^d}(A; B_{r_0}(x_0))$ is obtained by sandwiching $A$ between balls of $\Z^d$ with respect to the word metric, which for the canonical basis is the $\ell^1$ norm. But even as we let $r_0 \to 0$, the $\ell^1$ annulus that contains the $\ell^2$ sphere in $\R^d$ will have nonzero volume which will scale on the order of $R^d$. This prevents $\alpha$ from being replaced with $\alpha-1$ in Lemma \ref{lem:rgeqr0}. 

In such specific examples, possibly also for periodic metrics in general, one could estimate $N_\Gamma(A; B_{r_0}(x_0))$ or $N_\Gamma(A; K)$ more directly, which could lead to a more precise exponent in Lemma \ref{lem:rgeqr0}.
\end{remark*}

\appendix

\section{Relation between covering growth and Minkowski contents of spheres}\label{app:mincon}

The supremum on the right-hand side of equation \eqref{eq:supregbd} can be thought of as an approximate version of the upper perimeters or \emph{Minkowski contents} of sets:

\begin{defin}
Let $M$ be a $d$-dimensional Riemannian manifold, $A \subseteq M$ compact. Let $m \leq d$ be a natural number. The $m$-dimensional Minkowski content of $A$ is given by
\[ \M^m(A) = \lim_{r \to 0^+} \frac{\vol_M(A^r)}{\omega_{d-m} r^{d-m}}, \]
where the upper resp. lower limits are denoted by $\overline{\M}$ resp. $\underline{\M}$. 
\end{defin}

For $m$-rectifiable subsets $A \subseteq M$, $\M^m(A)$ is known to exist and coincide with the $m$-dimensional Hausdorff measure. Moreover, for geodesic spheres, the Minkowski content of the sphere $\partial B_R(x_0)$ can also be obtained as the derivative of $\vol_M(B_R(x_0))$ with respect to $R$. For reference on Minkowski contents and perimeters in the more general setting of metric measure spaces, see Ambrosio, Marino and Gigli \cite{Ambrosio:2017tf}. 

For reference, we mention that on Riemannian manifolds with $\nu = \vol_M$, the quantity $N(A; r) r^m$ becomes comparable to $\M^m(A)$ as $r \to 0^+$:

\begin{prop}\label{prop:mincon}
Let $M$ be a $d$-dimensional Riemannian manifold, $A \subseteq M$ compact, $m \leq d$.

Then there exists $C > 0$ dependent only on $d$ and $m$ such that
\[ C^{-1} \underline{\M}^m(A) \leq \liminf_{r \to 0^+} N(A; r) r^m \leq \limsup_{r \to 0^+} N(A; r) r^m \leq C \overline{\M}^m(A). \]
\end{prop}

\begin{lemma}\label{lem:limdens}
Let $M$ be a $d$-dimensional Riemannian manifold. Then
\[ \lim_{r \to 0^+} \frac{\vol_M(B_r(x))}{\omega_d r^d} = 1 \]
uniformly on compact subsets of $M$.
\end{lemma}
\begin{proof}
Let $A \subseteq M$ be compact. For each $\varepsilon > 0$, cover $A$ with finitely many normal coordinate charts $\{(U_i, \phi_i)\}_{i=1}^N$, small enough so that each chart $\phi_i \colon U_i \to \phi_i(U_i) \subseteq \R^d$ is $(1-\varepsilon, 1+\varepsilon)$-bi-Lipschitz.

By the Lebesgue number lemma, there exists $\delta > 0$ such that any ball of radius $r \leq \delta$ with center in $A$ is contained in a single chart $U_i$. 
Let $r \leq \delta$, $x \in A$, and suppose $B_r(x) \subseteq U_i$. Then 
\[ B_{(1-\varepsilon)r}(\phi_i(x)) \subseteq \phi_i(B_r(x)) \subseteq B_{(1+\varepsilon)r}(\phi_i(x)) \]
so that 
\[ \omega_d (1-\varepsilon)^d r^d \leq \L^d(\phi_i(B_r(x))) \subseteq \omega_d (1+\varepsilon)^d r^d. \]
Moreover, since the volume form on $M$ coincides with the $d$-dimensional Hausdorff measure, the fact that $\phi_i$ is $(1-\varepsilon, 1+\varepsilon)$-bi-Lipschitz implies
\[ (1-\varepsilon)^d \L^d(\phi_i(B_r(x))) \leq \vol_M(B_r(x)) \leq (1+\varepsilon)^d \L^d(\phi_i(B_r(x))). \]
Therefore
\[ (1-\varepsilon)^{2d} \leq \frac{\vol_M(B_r(x))}{\omega_d r^d} \leq (1+\varepsilon)^{2d}, \]
and this inequality holds for all $x \in A$ and all $0 < r \leq \delta$.
\end{proof}

\begin{proof}[Proof of Proposition \ref{prop:mincon}]
First note that by Proposition \ref{prop:supregbd},
\[ \sup_{r '\leq 2r_0} N(A; r) r^m \leq \lambda_d(A; r_0)^{-1} \sup_{r \leq r_0} \frac{\vol_M(A^r)}{r^{d-m}}. \]
Letting $r_0 \to 0^+$ and noting that $\lambda_d(A; r_0) \to \omega_d$ by Lemma \ref{lem:limdens} yields the latter inequality.

For the former inequality, fix $r_0 > 0$, and note that since $A^\prime := \overline{A^{r_0}}$ is compact,
\[ \lim_{r \to 0^+} \frac{\vol_M(B_r(x))}{\omega_d r^d} = 1 \]
uniformly on $A^\prime$.
For $r \leq r_0$ arbitrary, let $S$ be an $r$-cover of $A$ with $\# S = N(A; r)$. Then $S \subseteq A^{r_0}$ by the minimality of $S$, and $A^r \subseteq \bigcup_{x \in S} B_{2r}(x)$ so that
\[ \vol_M(A^r) \leq \sum_{x \in S} \vol_M(B_{2r}(x)) \leq N(A;r) \sup_{\substack{x \in A^\prime \\ r^\prime \leq 2 r_0}} \frac{\vol_M(B_{r^\prime}(x))}{\omega_d r^{\prime d}} \omega_d (2r)^d. \]
Rearranging and taking the infimum over all $r \leq r_0$,
\[ \inf_{r \leq r_0} \frac{\vol_M(A^r)}{\omega_{d-m} r^{d-m}} \leq \frac{2^d \omega_d}{\omega_{d-m}} \sup_{\substack{x \in A^\prime \\ r^\prime \leq 2 r_0}} \frac{\vol_M(B_{r^\prime}(x))}{\omega_d r^{\prime d}} \inf_{r \leq r_0} N(A; r) r^m. \]
Letting $r_0 \to 0^+$ then yields the former inequality.
\end{proof}

\section{Auxiliary lemmas}\label{app:auxlem}

\subsection{Pierce's lemma for floor quantization.}

In the proof of Theorem \ref{thm:quantub}, radii $r \in [0,\infty)$ are not assigned to the closest quantizer but to the largest quantizer below them; i.e., each point $r \in [r_i, r_{i+1})$ is mapped to $r_i$. In that case, the Pierce upper bound on the quantization error needs to be proven separately. We follow the random quantizer argument in Graf and Luschgy \cite[Lem. 6.6]{quantbook}.

\begin{defin}
Let $\mu$ be a finite Borel measure on $[0,\infty)$, $p \in [1,\infty)$. 
Given a set of quantizers $S = \{r_i\}_{i=1}^N \subseteq (0,\infty)$, indexed in ascending order, the induced \emph{floor quantizer} on $[0,\infty)$ assigns each $r \in [0,\infty)$ to the greatest element of $S$ bounded by $r$:
\[ F_S(r) := \max(\{0\} \cup S \cap [0,r]) = \left\{ \begin{matrix} 0, & r < r_1; \\ r_i, & r_i \leq r < r_{i+1}, \\ r_N, & r_N \leq r. \end{matrix} \right. \quad i = 1, \ldots, N-1; \]
The \emph{floor quantization cost} of $\mu$ of order $p$ with respect to $S$ is given by
\[ V^F_{p}(\mu; S) := \int_0^\infty |r-F_S(r)|^p \di \mu(r) = \sum_{i=0}^N \int_{r_i}^{r_{i+1}} (r-r_i)^p \di \mu(r), \]
where by convention $r_0 = 0$, $r_{N+1} = \infty$, and $\int_{r_i}^{r_{i+1}} = \int_{[r_i,r_{i+1})}$.

The \emph{floor quantization cost} of $\mu$ of order $p$ with cardinality $N$ is given by
\[ V^F_{N,p}(\mu) := \inf_{\# S \leq N} V^F_p(\mu; S). \]
\end{defin}

\begin{lemma}\label{lem:piercefloor}
Let $\mu$ be a finite measure on $[0,\infty)$, $p \in [1,\infty)$, $\delta > 0$.
Then for each $N \in \N$,
\[ N^p V^F_{N,p}(\mu) \leq C \int_0^\infty (1+r^{p+\delta}) \di \mu(r), \]
where $C = C(p,\delta) > 0$ is a constant independent of $\mu$.
\end{lemma}
\begin{proof}
Set $\beta := \delta/p$, and suppose the quantizers $S = \{r_i\}_{i=1}^N$ are generated independently according to the same Pareto distribution with cumulative distribution function
\[ G(y) := \left\{ \begin{matrix} 1 - (y+1)^{-\beta}, & y \geq 0; \\ 0, & y \leq 0. \end{matrix} \right. \]
Then by Tonelli's theorem, we can bound the quantization cost as follows:
\begin{align*}
V^F_{N,p}(\mu) \leq \E \int_0^\infty |r-F_S(r)|^p \di \mu(r) =  \int_0^\infty \E |r-F_S(r)|^p \di \mu(r),
\end{align*}
where the expectation corresponds to integration over $\di G(r_1) \ \ldots \di G(r_N)$.

Given $r,t \in [0,\infty)$, $|r-F_S(r)| = r - F_S(r) > t$ implies that $r > t$ and that for each $i = 1, \ldots, N$, either $r < r_i$ or $r > r_i + t$. Therefore
\begin{align*} 
\P(|r-F_S(r)| > t) 
& \leq \One(r > t) \prod_{i=1}^N \left( \P(r_i > r) + \P(r_i < r - t) \right) 
\\ & = \One(r > t) \left( 1 - G(r) + G(r-t) \right)^N 
\\ & = \One(r > t) \left(1 + (r+1)^{-\beta} - (r-t+1)^{-\beta} \right)^N.
\end{align*}
Considering $f(t) := (r-t+1)^{-\beta}$, we have $f^\prime(t) = \beta(r-t+1)^{-\beta-1} \geq \beta (r+1)^{-\beta-1}$ hence
\[ (r-t+1)^{-\beta} - (r+1)^{-\beta} = f(t) - f(0) \geq \beta (r+1)^{-\beta-1} t. \]
Consequently, for $0 < t < r$,
\[ \left(1 + (r+1)^{-\beta} - (r-t+1)^{-\beta} \right)^N \leq \left( 1 - \beta (r+1)^{-\beta-1} t \right)^N \leq \exp(-N \beta (r+1)^{-\beta-1} t). \]
Hence for $r \in [0,\infty)$,
\begin{align*}
\E |r-F_S(r)|^p & = p \int_0^\infty r^{p-1} \P(|r-F_S(r)| > t) \di t 
\leq p \int_0^r t^{p-1} \exp(-N \beta (r+1)^{-\beta-1} t) \di t.
\end{align*}
By the identity $\int_0^\infty t^{p-1} e^{-at} \di t = \Gamma(p) a^{-p}$ for $a > 0$, we have
\begin{align*}
\E |r-F_S(r)|^p 
& \leq p \int_0^\infty t^{p-1} \exp(-N \beta (r+1)^{-\beta-1} t) \di t 
\\ & = p \Gamma(p) \left( N \beta (r+1)^{-\beta-1} \right)^{-p} 
\\ & = \frac{\Gamma(p+1)p^{p+1}}{\delta^p} N^{-p} (r+1)^{p+\delta}
\leq C N^{-p} (r^{p+\delta} + 1)
\end{align*}
for $C := 2^{p+\delta-1} \Gamma(p+1)p^{p+1} \delta^{-p}$.
Integrating with respect to $r$ yields the statement.
\end{proof}

\subsection{Approximate perimeters in manifolds of constant curvature.}

For the prototypical manifolds of constant sectional curvature $\kappa \leq 0$, the volumes of geodesic balls and spheres satisfy the following bound:

\begin{lemma}\label{lem:pkbd}
Let $\kappa \leq 0$. There exist constants $C_1, C_2 \geq 0$ such that, for any $R \geq 2r > 0$,
\[ \frac{P^d_\kappa(R; r)}{\vol^d_\kappa(B_{R+r})} \leq \frac{\vol^d_\kappa(\partial B_{R+r})}{\vol^d_\kappa(B_{R+r})} \leq \frac{C_1 + C_2 \sqrt{-\kappa} (R+r)}{R+r}. \]
\end{lemma}
\begin{proof}
Note firstly that, since $\vol^d_\kappa(\partial B_R)$ is increasing with respect to $R$ for $\kappa \leq 0$, $P^d_\kappa(R; r) \leq \vol^d_\kappa(\partial B_{R+r})$. For $\kappa = 0$, we have
\[ \frac{\vol^d_\kappa(\partial B_{R+r})}{\vol^d_\kappa(B_{R+r})} = \frac{d \omega_d (R+r)^{d-1}}{\omega_d (R+r)^d} = \frac{d}{R+r}. \]
Now suppose $\kappa = -1$. Fix $r_0 > 0$. We investigate the cases $R \leq 2 r_0$ and $R \geq 2 r_0$ separately. Suppose $R \leq 2 r_0$, and let $r \leq R/2$ so that $R+r \leq 3r_0$. Then
\[ P^d_{-1}(R; r) \leq \vol^d_{-1}(\partial B_{R+r}) = d \omega_d \sinh^{d-1}(R+r) \]
and since $\sinh^\prime = \cosh$ is increasing on $[0,\infty)$,
\[ \sinh(R+r) \leq \sinh(0) + (R+r) \cosh(R+r) \leq (R+r) \cosh(3r_0) \]
so 
\[ P^d_{-1}(R;r) \leq d \omega_d \cosh^{d-1}(3r_0) (R+r)^{d-1}, \]
while $\vol^d_{-1}(B_{R+r}) \geq \vol^d_0(B_{R+r}) = \omega_d (R+r)^d$, so
\[ \frac{P^d_{-1}(R;r)}{\vol^d_{-1}(B_{R+r})} \leq \frac{d \cosh^{d-1}(3 r_0)}{R+r}. \]
Now suppose $R \geq 2 r_0$. Observe that
\[ e^{-t} \sinh(t) = \frac{1 - e^{-2t}}{2} \]
is increasing with respect to $t \in [0,\infty)$. Therefore
\begin{align*} 
\vol^d_{-1}(B_{R+r}) & \geq d \omega_d \int_{r_0}^{R+r} \sinh^{d-1}(t) \di t
\\ & \geq d \omega_d \left( e^{-r_0} \sinh(r_0) \right)^{d-1} \int_{r_0}^{R+r} e^{(d-1)t} \di t
\\ & = d \omega_d \left( e^{-r_0} \sinh(r_0) \right)^{d-1} \frac{e^{(d-1)(R+r)} - e^{(d-1)r_0}}{d-1}
\\ & \geq \omega_d \sinh^{d-1}(r_0) \left( e^{(d-1)(R+r-r_0)} - 1 \right),
\end{align*}
whereas
\[ P^d_{-1}(R; r) \leq d \omega_d \sinh^{d-1}(R+r) \leq d \omega_d e^{(d-1)r_0} e^{(d-1)(R+r-r_0)}. \]
Hence
\[ \frac{P^d_{-1}(R; r)}{\vol^d_{-1}(B_{R+r})} \leq \frac{d e^{(d-1)r_0}}{\sinh^{d-1}(r_0)} \frac{e^{(d-1)(R+r-r_0)}}{e^{(d-1)(R+r-r_0)}-1}, \]
and the right-hand side is uniformly bounded for $R+r-r_0 \geq r_0$.

We can therefore pick constants $C_1, C_2 \geq 0$ such that the inequality
\[ \frac{\vol^d_\kappa(\partial B_{R+r})}{\vol^d_\kappa(B_{R+r})} \leq \frac{C_1 + C_2 \sqrt{-\kappa} (R+r)}{R+r} \]
holds for $\kappa = 0, -1$. For other $\kappa < 0$, we have $\vol^d_\kappa(B_r) = \vol^d_{-1}(B_{\sqrt{-\kappa} r})$ and $\vol^d_\kappa(\partial B_r) = \sqrt{-\kappa} \vol^d_{-1}(\partial B_{\sqrt{-\kappa}r})$, and therefore
\[ \frac{\vol^d_\kappa(\partial B_{R+r})}{\vol^d_\kappa(B_{R+r})}
= \frac{\sqrt{-\kappa} \vol^d_{-1}(\partial B_{\sqrt{-\kappa}(R+r)})}{\vol^d_{-1}(B_{\sqrt{-\kappa}(R+r)})}
\leq \sqrt{-\kappa} \frac{C_1 + C_2 \sqrt{-\kappa} (R+r)}{\sqrt{-\kappa}(R+r)} = \frac{C_1 + C_2 \sqrt{-\kappa} (R+r)}{R+r}. \qedhere
\]
\end{proof}

\begin{paragraph}{{\bf Acknowledgments}}
The second author expresses gratitude to Beno\^it Kloeckner and Chiara Rigoni for their valuable discussions regarding this topic. We also extend our appreciation to Urs Lang for his insightful feedback and constructive comments on our work.
\end{paragraph}

\bibliographystyle{abbrv}
\nocite{*}
\raggedright
\addcontentsline{toc}{section}{References}
\bibliography{refs}

\end{document}